\documentclass[a4paper, fleqn, reqno, 10pt]{amsart}

\usepackage{times,latexsym,amssymb, stmaryrd}
\usepackage{amsmath,amsthm,constants}
\usepackage[active]{srcltx}
\usepackage[colorlinks,pdfpagelabels,pdfstartview = FitH,bookmarksopen = true,bookmarksnumbered = true,linkcolor = black,plainpages = false,hypertexnames = false,citecolor = black] {hyperref}

\numberwithin{equation}{section}

\newtheorem{Theorem}{Theorem}[section]

\newtheorem{Lemma}[Theorem]{Lemma}

\newcommand{\divo}{\textnormal{div}}

\def\dx{\,dx}

\def\tb{\tilde B_j}
\newcommand{\rif}[1]{(\ref{#1})}
\newcommand{\trif}[1] {\textnormal{\rif{#1}}}

\DeclareMathOperator*{\osc}{osc}

\def\mean#1{\mathchoice%
         {\mathop{\kern 0.2em\vrule width 0.6em height 0.69678ex depth -0.58065ex
                 \kern -0.8em \intop}\nolimits_{\kern -0.4em#1}}%
         {\mathop{\kern 0.1em\vrule width 0.5em height 0.69678ex depth -0.60387ex
                 \kern -0.6em \intop}\nolimits_{#1}}%
         {\mathop{\kern 0.1em\vrule width 0.5em height 0.69678ex
             depth -0.60387ex
                 \kern -0.6em \intop}\nolimits_{#1}}%
         {\mathop{\kern 0.1em\vrule width 0.5em height 0.69678ex depth -0.60387ex
                 \kern -0.6em \intop}\nolimits_{#1}}}

\newcommand{\aveint}[2]{\mathchoice%
          {\mathop{\kern 0.2em\vrule width 0.6em height 0.69678ex depth -0.58065ex
                  \kern -0.8em \intop}\nolimits_{\kern -0.45em#1}^{#2}}%
          {\mathop{\kern 0.1em\vrule width 0.5em height 0.69678ex depth -0.60387ex
                  \kern -0.6em \intop}\nolimits_{#1}^{#2}}%
          {\mathop{\kern 0.1em\vrule width 0.5em height 0.69678ex depth -0.60387ex
                  \kern -0.6em \intop}\nolimits_{#1}^{#2}}%
          {\mathop{\kern 0.1em\vrule width 0.5em height 0.69678ex depth -0.60387ex
                  \kern -0.6em \intop}\nolimits_{#1}^{#2}}}

\newcommand{\parenthezises}[1]{\arabic{#1}}
\newconstantfamily{c}{
symbol=c,
format=\parenthezises
}

\newcommand{\loc}{\textnormal{loc}}
\newcommand{\dist}{\textnormal{dist}}

\newcommand{\R}{{\mathbb R}}

\newcommand{\N}{{\mathbb N}}

\newcommand{\vs}{\vspace{3mm}}
\newcommand{\F}{{\mathcal F}}
\newcommand{\eps}{\varepsilon}

\def\er{\mathbb R}
\newcommand{\ern}{\mathbb{R}^n}

\def\eqn#1$$#2$${\begin{equation}\label#1#2\end{equation}}

\allowdisplaybreaks
\begin{document}

\title[Borderline gradient continuity of minima]{Borderline gradient continuity of minima}

\vspace{5mm}

\author{Paolo Baroni}
\address{Paolo Baroni, Department of Mathematics, Uppsala University, L\"agerhyddsv\"agen 1, SE-751 06, Uppsala, Sweden} \email{paolo.baroni@math.uu.se}

\author{Tuomo Kuusi}
\address{Tuomo Kuusi, Aalto University, Institute of Mathematics, P.O. Box 11100, FI-00076 Aalto, Finland} \email{tuomo.kuusi@tkk.fi}

\author{Giuseppe Mingione}
\address{Giuseppe Mingione, Dipartimento di Matematica e Informatica, Universit\`a di Parma, 
Parco Area delle Scienze 53/a, Campus, I-43124 Parma, Italy}
\email{giuseppe.mingione@unipr.it.}


\date{\scriptsize \today}

\maketitle
\centerline{{\em To Ha\"im Brezis, a master of nonlinear analysis}}

\begin{abstract}
The gradient of any local minimiser of functionals of the type
$$
w \mapsto \int_\Omega f(x,w,Dw)\dx+\int_\Omega w\mu\dx,
$$
where $f$ has $p$-growth, $p>1$, and $\Omega \subset \mathbb R^n$, is continuous provided the optimal Lorentz space condition $\mu \in L(n,1)$ is satisfied and $x\to f(x, \cdot)$ is suitably Dini-continuous.
\end{abstract}


\section{Introduction}
The aim of this paper is to prove a variational analog of a borderline regularity result that, originally being a well-known consequence of a linear fact due to Stein, has recently found optimal nonlinear extensions in \cite{KMI1, KMri, KMStein}. Indeed, a result from \cite{stein} asserts that if $v \in W^{1,1}$ is a Sobolev function defined in $\er^n$ with $n \geq2$, then 
\eqn{stein}
$$
Dv \in L(n,1)   \ \Longrightarrow \   v \ \mbox{is continuous}\,.
$$
The Lorentz space $L(n,1)(\Omega)\equiv L(n,1)$, here with $\Omega \subset \er^n$, consists of all measurable functions $g$ satisfying
\eqn{lorentzdefi}
$$
\int_0^\infty |\{x \in \Omega\, : \, |g(x)|>t\}|^{1/n} \, dt < \infty\;.
$$
The implication in \rif{stein} is the optimal limiting case of Sobolev-Morrey embedding. Applications to regularity actually lead to a reformulation of Stein's theorem in terms of sharp gradient continuity criteria for solutions to the linear equation
\eqn{lap}
$$
\triangle u =\mu\qquad \mbox{in} \quad \er^n\;.
$$
Indeed, using standard interpolation \rif{stein} allows to conclude that 
\eqn{mainimpli}
$$\mu \in L(n,1) \Longrightarrow Du\  \mbox{is continuous}\;.$$
Now, while this result seems to be deeply linked to the fact that the one in \rif{lap} is a linear equation, very recent developments in \cite{KMStein, KMI1} have revealed an unsuspected nonlinear nature of this phenomenon. Indeed, when considering the so called $p$-Laplacean equation defined by 
\eqn{plap}
$$
\divo\, (|Du|^{p-2}Du)=\mu,\qquad p>1\;,
$$ 
we see that, notwithstanding the nonlinear, degenerate nature of the operator appearing on the left-hand side, the implication in \rif{mainimpli} still holds. In fact, the result also extends to more general equations in divergence form with $p$-Laplacean structure as
\eqn{plap2}
$$
\divo\, a(x,Du)=\mu\;,
$$ 
where the vector field $x \mapsto a(x, \cdot)$ is Dini-continuous (see \rif{Dini1} below and \cite{KMI1, KMri} for precise assumptions). The Dini-continuity of coefficients is a necessary conditions 
for the continuity of the gradient already in the case of linear equations, as shown in \cite{mazyaex}. One of the most interesting features of the condition in \rif{mainimpli} is that this is actually independent of $p$. The one in \rif{plap} is the Euler-Lagrange equation of the functional
\begin{equation}\label{functional0}
w \mapsto \frac1p\int_\Omega |Dw|^p  \dx+\int_\Omega w\mu\dx
\end{equation}
and this observation is the starting point of this paper. Indeed, here we are interested in understanding 
if a result of the type explained above has also a variational nature. For this reason we are considering functionals of the type
\begin{equation}\label{functional}
\mathcal F(w,\Omega):=\int_\Omega f(x,w,Dw)\dx+\int_\Omega w\mu\dx\;.
\end{equation}
Here $\Omega$ denotes a bounded domain of $\R^n$, $n\geq2$ and the functional is naturally defined over the space $W^{1,p}(\Omega)$, $p>1$, having the integrand $f:\Omega\times\R\times\R^n\to\R$ polynomial growth of order $p$ with respect to the gradient variable. Indeed, the growth and ellipticity assumptions we impose on the energy density $f:\Omega\times\R\times\R^n\to\R$ are the classic
\begin{equation}\label{asss.f}
\hspace{-5mm}\begin{cases}
\xi\mapsto f(x,u,\xi)\qquad\text{is $C^2$-regular}\\[3pt]
\nu(s+|\xi|)^p\leq f(x,u,\xi)\leq L(s+|\xi|)^p \\[3pt]
\nu(s+|\xi|)^{p-2}|\lambda|^2\leq \langle \partial^2 f(x,u,\xi)\lambda,\lambda\rangle\leq L(s+|\xi|)^{p-2}|\lambda|^2\\[3pt]
|f(x_1,u_1,\xi)-f(x_2,u_2,\xi)| \leq \tilde L\big[\omega(|x_1 -x_2|)+|u_1-u_2|^\alpha\big](s+|\xi|)^p
\end{cases}
\end{equation}
for every $x,x_1,x_2\in\Omega$, every $u,u_1,u_2\in\R$, $\xi,\lambda\in\R^n$, for parameters $0<\nu\leq 1\leq L$, $\tilde L\geq0$, $s\in[0,1]$ and $\alpha\in(0,1]$. Here, and also later in the manuscript, $\partial f$ stands for the gradient of $f$ with respect to the $\xi$ variable and $\omega\colon [0,\infty) \to [0,1]$ is a {\em modulus of continuity}, that is a continuous, non-decreasing, concave function such that $\omega(0)=0$. These conditions are obviously satisfied by the functional in \rif{functional0} with $s=0$ and with no dependence on $(x,u)$ occurring, i.e. $\tilde L =0$. 

\vspace{3mm}

We recall that a {\em local minimiser} of the functional $\mathcal F$ is a map $u\in W^{1,p}_\loc(\Omega)$ such that
\[
\mathcal F(u,{\rm supp}\, \varphi)\leq \mathcal F(u+\varphi,{\rm supp}\, \varphi)
\]
for any variation $\varphi\in W^{1,p}(\Omega)$ such that ${\rm supp}\, \varphi\subset \Omega$. The main point here is that the functional $\F$ is non-differentiable, due to the H\"older continuity dependence with respect to the second variable displayed in \rif{asss.f}$_4$ (at least when $\alpha <1$). In other words no Euler-Lagrange equation is available for $\F$ and the results of \cite{KMI1, KMri, KMStein} cannot be used. The aim of this paper is now to show that the one in \rif{mainimpli} is in fact a general phenomenon, that does not require an equation, and that holds directly for minimisers of non necessarily differentiable functionals. The results will also depend, essentially in an optimal way, on the regularity of the integrand $f(\cdot)$ with respect to the variable $x$, and for this we need a preliminary definition. Indeed, we start assuming that $\omega(\cdot)$ appearing in \rif{asss.f}$_4$ 
is $1/2$-Dini-continuous, i.e.
\begin{equation}\label{Dini1/2}
\int_0[\omega(\rho)]^{1/2}\,\frac{d\rho}{\rho}<\infty\;.
\end{equation}
Note that this condition is weaker than the H\"older regularity, but stronger than the plain {\em Dini continuity}, which is in fact 
\begin{equation}\label{Dini1}
\int_0\omega(\rho)\,\frac{d\rho}{\rho}<\infty\;.
\end{equation}
Our first result is now
\begin{Theorem}\label{main.thm}
Let $u\in W^{1,p}_\loc(\Omega)$ a local minimiser of the functional \eqref{functional}, where the energy density $f(\cdot)$ satisfies assumptions \eqref{asss.f}, $\omega(\cdot)$ is $1/2$-Dini continuous in the sense of \trif{Dini1/2} and where $\mu\in L(n,1)$. Then $Du$ is continuous.
\end{Theorem}
The result of the previous theorem is the analog of the one from \cite{KMI1, KMri} valid for general divergence form equations of the type in \rif{plap2} apart from the fact that $1/2$-Dini continuity \rif{Dini1/2} is required instead of the weaker Dini continuity \rif{Dini1}. This fact is not technical. In fact, as noticed in \cite{GGb}, the modulus of continuity of the function $x \to f(x, \cdot)$ is not in general inherited by 
$x \to \partial f(x, \cdot)$ and results cannot be recovered by using the Euler-Lagrange equation, also in the case this last one exists. For instance, it can be proved that under assumptions \rif{asss.f} if $x \to f(x, \cdot)$ is H\"older continuous 
with exponent $\alpha \in (0,1)$ then $x \to \partial f(x, \cdot)$ is H\"older continuous with a worst exponent, namely $\alpha/2$. 
Such loss of regularity appears with other moduli of continuity too. In this respect, and recalling that the Dini continuity is in general necessary for proving the gradient continuity of solutions to equations as \rif{plap2}, the assumption of $1/2$-Dini continuity of Theorem \ref{main.thm} appears to be optimal in that it serves to rebalance this loss of regularity when passing from $f$ to $\partial f$. In order to use Dini-continuity of coefficients as an effective assumption we then need to consider an additional, natural condition on $\partial f (\cdot)$, namely 
\begin{multline}\label{further.asss.f}
|\partial f(x_1,u_1,\xi)-\partial f(x_2,u_2,\xi)| \\\leq \tilde L\big[ \omega(|x_1 -x_2|)+|u_1-u_2|^{\alpha}\big](s+|\xi|)^{p-1}
\end{multline}
for every $x_1,x_2\in\Omega$, every $u_1,u_2\in\R$, $\xi\in\R^n$; this assumptions is also of common use in the literature \cite{GGb, KMBoundary}. In this way we have 
\begin{Theorem}\label{second.thm}
Let $u\in W^{1,p}_\loc(\Omega)$ a local minimiser of the functional \eqref{functional} and suppose that $f(\cdot)$ satisfies \eqref{asss.f} and also \eqref{further.asss.f}, where $\omega(\cdot)$ is Dini continuous as in \trif{Dini1}, with $\mu\in L(n,1)$. Then $Du$ is continuous.
\end{Theorem}
Assumption \rif{further.asss.f} above is automatically satisfied in many cases, for instance by splitting densities 
and functionals of the type
\[
w \mapsto \int_\Omega g(x,w)h(Dw)\dx+\int_\Omega w\mu\dx\;,
\]
with 
\[
|g(x_1,u_1)-g(x_2,u_2)|\leq c\big[ \omega(|x_1 -x_2|)+|u_1-u_2|^{\alpha}\big]
\] 
and $h:\R^n\to\R$ having $p$-growth in the sense of \eqref{ass.h}$_1$ below. In this respect, 
Theorem \ref{second.thm} recovers the results of \cite{KMI1, KMri} for equations as in \rif{plap2}, when these are the Euler-Lagrange of a functional of the type considered here. This happens for instance when considering differentiable functionals as
\[
w \mapsto \int_\Omega f(x,Dw)\dx+\int_\Omega w\mu\dx
\]
which is in fact \rif{plap2} with $a=\partial f$, when \eqref{further.asss.f} is in force. 

\vs

We conclude spending a few words on the techniques used in this paper and the relative background. 
The methods we use find their origins in nonlinear potential theory \cite{MH, KL1, KL2, minjems, DM} and use certain exit time arguments and linearisation methods already introduced in \cite{KMI1, KMStein} together with several basic regularity results for solutions to $p$-Laplacean type equations (see for instance \cite{DB, Manfredi1, Manfredi2}). The novelty here consists of framing these recently introduced techniques in the variational setting, and performing estimates without using equations but using directly the minimality 
property. In this respect our results can also be framed in a line of research that started with the papers \cite{GGiusti, GGInv} and that has gained several contributions (see for instance \cite{KMfunctionals, dark} for references). The common point in these papers 
is in fact that regularity results are obtained using directly the minimality property of solutions rather than the fact that they solve an equation. 
In particular, in this paper we build a bridge between these variational techniques and nonlinear potential theory, thereby proving some borderline results for minimisers of non-differentiable functionals. Finally, a few words on the role of the space $L(n,1)$. 
This space already appears in the study of the $p$-Laplacian equations and systems \cite{CM1, CM2, DM0, KMStein, guide}. Lorentz spaces are of common use to prove endpoint estimates and describe results that are otherwise 
unachievable using Lebesgue spaces \cite{BW, BFS}. 

\section{Preliminary material}

\subsection{Notation}
In this paper we shall adopt the convention of denoting by $c$ a constant, {\em always larger than one}, that may vary from line to line; peculiar dependencies on parameters will be properly emphasized in parentheses when needed, sometimes just at the end of the chains of equations, for the sake of readability. Special occurrences will be denoted by special symbols, such as $c_1,c_2, \tilde c$. In the following \[B_R(x_0):=\{x \in\er^n \, :\, |x-x_0|< R\}\] will denote the open ball with center $x_0$ and radius $R$. We shall avoid to write the center of the balls when no ambiguity will arise: often the reader will read $B_R \equiv B_R(x_0)$ or the like. With $\delta$ being a positive number, we shall also denote by $\delta B$ the ball concentric to $B$ with radius magnified by a a factor $\delta$. With $\mathcal B\subset \ern$ being a measurable set with positive, finite measure and $\ell:\mathcal B \to \R^k$, $k \in \N$, an integrable map, we denote with ${(\ell)}_{\mathcal B}$ the averaged integral
\[
      {(\ell)}_{ \mathcal B}:=\mean{ \mathcal B} \ell\, dx := \frac{1}{|\mathcal B|} \int_{\mathcal B} \ell\, dx\;.
\]
A useful property, which will be often used, is the following one:
\begin{equation}\label{prop.exc}
\mean{\mathcal B}|\ell-(\ell)_{\mathcal B}|^t\dx\leq 2^t\mean{\mathcal B}|\ell-\xi|^t\dx \qquad\text{for all $t\geq 1$ and all $\xi\in\R^k$} ;
\end{equation}
here $\mathcal B$ is as above and $\ell\in L^t(\mathcal B)$. With $f \colon \mathcal B \to \er^k$ being a vector field, we shall denote
\[
\osc _{\mathcal B} f := \sup_{x,y\in \mathcal B} \, |f(x)-f(y)|\;,
\]
where with $\sup$ we denote the essential supremum. Finally, $\N:=\{1,2,\dots\}$ while $\N_0:=\N\cup\{0\}$.

\subsection{Lorentz spaces}\label{Lorentz}
The Lorentz space $L(n,1)$ has already been defined in \rif{lorentzdefi} to describe the main assumption concerning the function $\mu$ appearing in \rif{functional}. By eventually letting $\mu \equiv 0$ outside $\Omega$ we may assume that $\mu$ is defined on the whole $\ern$ and that 
\[
\int_0^\infty |\{x \in \ern\, : \, |\mu(x)|>t\}|^{1/n} \, dt < \infty\;.
\]
Therefore from now on we shall denote $L(n,1)\equiv L(n,1)(\er^n)$. One useful quantity related to the Lorentz space $L(n,1)$ is the following series
\begin{equation}\label{serie}
S_{r,\delta,q}(x_0):=\sum_{j=0}^\infty \delta^j r\Big(\mean{B_{\delta^j r}(x_0)}|\mu|^q\dx\Big)^{1/q}\qquad\text{for $q\in(1,n)$}, 
\end{equation}
where $\delta\in (0,1)$. The series is converging in the case $\mu\in L(n,1)$; the relation is encoded in the following Lemma, whose simple proof follows from the representation of Lorentz spaces in term of rearrangements (see \cite{Hardy}) and can be found in \cite[Lemma 1]{KMStein}.
\begin{Lemma}\label{Lem.Ln1}
Let $\mu\in L(n,1)$ be such that $\mu \equiv 0$ outside $\Omega$; then, if $\delta\in(0,1/4)$ and for $q\in(1,n)$, it holds that
\[
 \sup_{x_0 \in \Omega} \, S_{r,\delta,q}(x_0)\leq c\int_0^\infty |\{x \in \ern\, : \, |\mu(x)|>t\}|^{1/n} \, dt
\]
for a constant $c$ depending only on $n,q,\delta$. Moreover, there exists a continuous, non-decreasing function $d \colon [0, \infty) \to [0, \infty)$, obviously depending on $\mu$, such that $d(0)=0$ and the following inequality holds uniformly with respect to $x_0 \in \Omega$:
\begin{equation}\label{Sq}
S_{r,\delta,q}(x_0)\leq \delta^{1-n/q}d(r)\;.
\end{equation}
In particular, the following limit holds uniformly with respect to $x_0\in \Omega$:
\begin{equation}\label{Sq2}
\lim_{r\to 0}\, S_{r,\delta,q}(x_0)=0\;.
\end{equation}
\end{Lemma}

\subsection{The $V$ function}
It is useful to consider the following auxiliary map:
\[
V_s(\xi):=\big(s+|\xi|\big)^{(p-2)/2}\xi,\qquad \xi\in\R^n,
\]
$s$ as in \eqref{asss.f}, which is a bijection of $\R^n$. The following inequality is classic:
\begin{multline}\label{equiv.V}
\frac1c|V_s(\xi_1)-V_s(\xi_2)|^2\leq\big(s+|\xi_1|+|\xi_2|\big)^{p-2}|\xi_1-\xi_2|^2\\
\leq c|V_s(\xi_1)-V_s(\xi_2)|^2
\end{multline}
for any $p>1$ and a constant $c$ depending on $n,p$; in particular
\begin{equation}\label{equiv.V.p}
|\xi_1-\xi_2|^p\leq c|V_s(\xi_1)-V_s(\xi_2)|^2
\end{equation}
when $p\geq2$, while 
\begin{equation}\label{equiv.V.p-}
|\xi_1-\xi_2|\leq c|V_s(\xi_1)-V_s(\xi_2)|^{2/p}+c\big(s+|\xi_1|\big)^{(2-p)/2}|V_s(\xi_1)-V_s(\xi_2)|
\end{equation}
if $1<p\leq 2$, see \cite[Lemma $2$]{KMStein}. Also the constants appearing in \eqref{equiv.V.p} and \eqref{equiv.V.p-} depend only on $n$ and $p$.

It will be also useful the following fact, which can be deduced by Taylor's formula, using \eqref{asss.f}$_3$ and \eqref{equiv.V}: see for instance \cite[(3.2)]{KMfunctionals}.
\begin{Lemma}\label{Lemma.V}
Suppose $f(\cdot)$ satisfies \eqref{asss.f}$_1$ and \eqref{asss.f}$_3$, with $p>1$. Then for all $x,u\in \Omega\times\R$ and for all $\xi_1,\xi_2\in \R^n$ it holds
\[
\frac1c\big|V_s(\xi_1)-V_s(\xi_2)\big|^2\leq f(x,u,\xi_1)-f(x,u,\xi_2)-\langle \partial f(x,u,\xi_2),\xi_1-\xi_2\rangle
\]
for a constant depending only on $n,p,\nu, L$.
\end{Lemma}

\subsection{Regularity estimates for minimisers}
We start with a suitable reformulation of some standard estimates for minimisers that can be found for instance in \cite[Chapter 7]{G}. These are summarised in the following
\begin{Theorem}\label{Holder}
Let $u$ be a minimiser of \eqref{functional}; then $u$ is locally H\"older continuous. Moreover, there exists a radius $R_H\equiv R_H(n,p,\nu,L,\|\mu\|_{L^n})\leq1$  such that the estimate
\begin{equation}\label{osc.u}
\osc_{B_\rho(x_0)}u\leq c\left(\frac{\rho}{R}\right)^\gamma\left(\mean{B_{2R}(x_0)}\big(|u|+1\big)^p\dx\right)^{1/p},\qquad 0<\rho\leq R,
\end{equation}
holds whenever $B_{2R}(x_0)\subset\Omega$ is a ball with radius $2R\leq R_H$, for an exponent $\gamma\equiv \gamma(n,p,\nu,L)\in(0,1)$ and a constant $c\geq1$ depending on $n,p,\nu,L,\|\mu\|_{L^n}$.
\end{Theorem}
\begin{proof}
We briefly show how to deduce \eqref{osc.u} from the estimates of \cite[Chapter 7]{G} and we try to stick to the notation used there. We can reduce to the case $p<n$, as noted in \cite[Chapter 7]{G} and reformulate the growth conditions we have now in terms of those used in \cite[Chapter 7]{G}. When looking at the functional in \rif{functional} we see that the integrand globally satisfies the growth conditions
\eqn{crescite}
$$
\nu|\xi|^p-|\mu||w|^p-|\mu|\leq f(x,w,\xi)+w\mu\leq c|\xi|^p+c|\mu||w|^p+c(|\mu|+s^p)
$$
for a.e. $x\in\Omega$ and all $(w,\xi)\in\R^{n+1}$; the constant $c$ only depends on $n,p,L$. Hence in \cite[Equation ($7.2$)]{G} we can take $\gamma=p$, $b(x):=c|\mu(x)|$ and $a(x):=c(|\mu(x)|+s^p)$; they both clearly belong to $L^n(\Omega)$ and this in enough to ensure the needed integrability required in \cite[Chapter 7]{G} with respect to the lower order terms. A computation moreover shows that we can take $\epsilon=(p-1)/n$ and hence $\beta=n\epsilon/p=1-1/p$ in all the results of \cite[Chapter 7]{G}. In particular, from \cite[Theorem $7.5$]{G} we infer that $u$ is locally bounded and that the estimate
\begin{equation}\label{sup.bound}
\sup_{B_R(x_0)}u\leq c\Big(\mean{B_{2R}(x_0)}|u|^p\dx\Big)^{1/p}+c\||\mu|+s^p\|^{1/p}_{L^n(B_{2R}(x_0))}R^{1-1/p} 
\end{equation}
holds for a constant depending on $n,p,\nu, L$ and for any ball $B_{2R}(x_0)\subset\Omega$ with radius smaller than a threshold $R_H$, depending on $n,p,\|\mu\|_{L^n}$; notice indeed that by the definition of $\xi(R)$ at page $216$, in the case $\gamma=p$ the smallness condition involves only $\|\mu\|_{L^n}$ and we can suppose $R_H\leq1$. In other words, the dependence of the radius $R_H$ on the norm $\|u\|_{W^{1,p}}$ described in \cite[Chapter 7]{G} does not take place in the present situation. One can deduce the local estimate \eqref{sup.bound} also starting from \cite[Display ($7.16$)]{G}, where the value of $\chi$ must be chosen accordingly with the Caccioppoli's inequality \cite[Display ($7.5$)]{G} and $\kappa_0=0$; at this point \eqref{sup.bound} follows performing some simple algebraic manipulations. Now we consider \cite[Display ($7.45$)]{G} that states
$$
\osc_{B_\rho(x_0)}u\leq c\Big(\frac{\rho}{R}\Big)^\gamma\Big[\osc_{B_R(x_0)}u+\chi R^\gamma\Big],\quad \chi^p:=\big\||\mu|(1+s^p+M^p)\big\|_{L^n(B_R(x_0))},
$$
for all $\rho\leq R\leq R_H$, $\gamma\in(0,1)$ as in the statement of the Theorem \ref{Holder} and a constant $c\geq1$, depending both on $n,p,\nu,L$ and where $M\geq \|u\|_{L^\infty(B_R(x_0))}$. To conclude, we estimate part of the right-hand side of the previous inequality, using \eqref{sup.bound}, in the following manner: since $R\leq R_H\leq 1$
\begin{align*}
\osc_{B_R(x_0)}u+\chi R^\gamma&\leq 2M+c(1+s+M)\|\mu\|_{L^n(B_R(x_0))}^{1/p}\\
&\leq c\big(1+\|\mu\|_{L^n(B_R(x_0))}^{1/p}\big)(1+s+M)\\
&\leq c\left(\mean{B_{2R}(x_0)}\big(|u|+1\big)^p\dx\right)^{1/p}.
\end{align*}
where $c$ also depends on $\|\mu\|_{L^n}$. 
\end{proof}
\subsection{Regularity estimates for frozen functionals}
In this section we collect a few standard facts from regularity theory of $p$-Laplacean type equations and related variational integrals. We consider variational Dirichlet problems of the type
\begin{equation}\label{M.G}
\left\{
    \begin{array}{c}
   \displaystyle  v \to \min_w
\, \int_{B_R} h(Dw)\dx   
\\[10pt]
        w\in u+W^{1,p}_0(B_R)\;,
\end{array}\right.
\end{equation}
where $B_R$ is a ball of $\R^n$, $n\geq2$, $u\in W^{1,p}(B_R)$ and the energy density $h(\cdot)$ is a $C^2(\R^n)$ map satisfying
\begin{equation}\label{ass.h}
\begin{cases}
\nu(s+|\xi|)^p\leq h(\xi)\leq L(s+|\xi|)^p; \\[3pt]
\nu(s+|\xi|)^{p-2}|\lambda|^2\leq \langle \partial^2 h(\xi)\lambda,\lambda\rangle\leq L(s+|\xi|)^{p-2}|\lambda|^2,
\end{cases}
\end{equation}
for all $\xi,\lambda\in \R^n$ where $s\in[0,1]$. We stress that by using \eqref{ass.h}$_1$ and the convexity of $h$ it is possible to show that
\begin{equation}\label{grad.h.growth}
|\partial h(\xi)|\leq c(p,L)(s+|\xi|)^{p-1}\qquad\text{for all $\xi\in\R^n$} 
\end{equation}
see \cite[Proposition 2.32]{Dacorogna} and hence $w$ solves the Euler equation relative to the functional \eqref{M.G} in its weak formulation, that is
\begin{equation}\label{euler.equation.h}
\int_{B_R}\langle \partial h(Dv),D\varphi\rangle\dx=0\qquad\text{for all $\varphi\in W^{1,p}_0(B_R)$}. 
\end{equation}
Hence the vector field $\partial h $ satisfies the ellipticity and growth conditions
\[
 \langle \partial^2 h(\xi)\lambda,\lambda\rangle\geq\nu(s+|\xi|)^{p-2}|\lambda|^2,\qquad |\partial h(\xi)|\leq c(p,L)(s+|\xi|)^{p-1}
\]
for any $\xi,\lambda\in\R^n$. The following two lemmas are a direct consequence of the classical regularity estimates valid for $p$-Laplacean type equations \cite{DB, Manfredi1, Manfredi2}. For the formulations below see for instance \cite{KMI1, guide}. 
\begin{Lemma}\label{Lem.regolarity}
Let $v\in W^{1,p}(B_R)$ be the unique solution to \eqref{M.G}; then $v\in C^{1,\beta}_\loc(B_R)$ for some $\beta(n,p,\nu, L)\in(0,1)$. Moreover the following local estimates hold:
\begin{equation}\label{sup.lambda}
\sup_{B_{R/2}}\, (s +|Dv|)\leq \Cl[c]{sup.est}\mean{B_R}\big(s+|Dv|\big)\dx 
\end{equation}
and, with $\tilde B,\lambda\geq1$, 
\begin{equation}\label{osc.lambda}
s+\sup_{B_{R/4}}|Dv|\leq \tilde B\lambda\qquad\Longrightarrow\qquad\osc_{B_{\tau R}}Dv\leq \Cl[c]{osc.est}\tau^\beta
 \tilde B\lambda 
\end{equation}
for all $\tau\in (0,1/8)$; the constants $\Cr{sup.est},\Cr{osc.est}$ depend only on $n,p,\nu, L$.
\end{Lemma}
\begin{Lemma}\label{ex.deca}
Let $v\in W^{1,p}(B_R)$ be the unique solution to \eqref{M.G}. Then for every $\bar\eps\in(0,1)$, there exists $\bar\delta\in(0,1/4)$, depending on $n,p,\nu, L,\bar\eps$, such that if $\sigma\in(0,\bar\delta\,]$, then
\begin{equation}\label{exc.smallness}
\mean{B_{\sigma R}}\big|Dv-(Dv)_{B_{\sigma R}}\big|^t\dx\leq \bar\eps \mean{B_R}\big|Dv-(Dv)_{B_R}\big|^t\dx 
\end{equation}
for any $t\in[1,2]$.
 \end{Lemma}
A corollary of Lemma \ref{Lem.regolarity}, is the following ``density improvement Lemma'' first used in \cite{KMI1}; we refer to \cite[Proposition 2]{KMI1} for the proof (there it is $b=0$ but the proof applies verbatim to the case considered below).
\begin{Lemma}\label{Lemmatau}
Let $v\in W^{1,p}(B_R)$ be as in Lemma \ref{Lem.regolarity} and suppose that the two conditions
\begin{equation}\label{density.hp}
 \frac{\lambda}{\tilde A}\leq\left(\mean{B_{\sigma{\tau} R}}(b+|Dv|)^t\dx\right)^{1/t}\qquad\text{and}
 \qquad s+\sup_{B_{R/2}}|Dv|\leq \tilde A\lambda,
\end{equation}
hold for $t\in [1,p]$, $b\in[0,1]$, some $\sigma \in (0,1)$, $\tilde A\geq1$, $\lambda\geq1$ and with $\tau\in (0,\bar\tau]$, satisfying
\begin{equation}\label{sigma}
\bar\tau:=\frac12\frac{1}{[4\Cr{osc.est}\tilde A^2]^{1/\beta}}\in(0,1/8)\;.
\end{equation}
Here both $\beta\in(0,1)$ and $\Cr{osc.est}$ appear in Lemma \ref{Lem.regolarity}. Then
\[
\frac{\lambda}{4\tilde A}\leq b+|Dv|\qquad \text{holds in $B_{\tau R}$\;.}
\]
\end{Lemma}

\section{Proof of Theorems \ref{main.thm}-\ref{second.thm}: beginning}\label{beg sec}
We begin the proof of Theorems \ref{main.thm}-\ref{second.thm}. Before going on, let us make a few preliminary remarks. We can define the new modulus of continuity
\eqn{tiome}
$$
\tilde \omega(\varrho) :=  \omega(\varrho) + \varrho^{\alpha\gamma}\,,
$$
$\alpha$ in \eqref{asss.f}$_4$, $\gamma$ in Theorem \ref{Holder} and notice that if $\omega(\cdot)$ is $1/2$-Dini continuous, then also $\tilde \omega(\cdot)$ is $1/2$-Dini continuous. Let then $\Omega '\Subset \Omega$ be an open subset. Using Theorem \ref{Holder} and a standard covering argument, we find that
$u \in C^{0,\gamma}(\Omega')\cap L^{\infty}(\Omega')$  for $\gamma\equiv \gamma (n,p,\nu, L)\in (0,1)$ and we have the inequality
\eqn{tiome2}
$$
\big[\osc_{B_\rho} u\big]^\alpha\leq c \tilde \omega(\varrho)
$$
for a constant $c\equiv c (n,p,\nu, L, \|\mu\|_{L^n},\|u\|_{L^p}(\Omega),\dist(\Omega', \partial \Omega))$, which is of later frequent use. Now since the results we are going to prove are local in nature, we can therefore assume without loss of generality that 
\eqn{globalass}
$$u \in C^{0,\alpha}(\Omega)\cap L^{\infty}(\Omega)$$
holds, getting rid of the dependence of the various constants on $\|u\|_{L^p(\Omega)}, \dist(\Omega', \partial \Omega)$, and just retaining the ones on 
$n,p,\nu, L, \|\mu\|_{L^n}$. 

We shall now consider a fixed ball $B_{2R}\equiv B_{2R}(x_0)\subset\Omega$ with radius $2R\leq R_H$, $R_H$ appearing in Theorem \ref{Holder}. The scheme of the section is now the following. In Sections \ref{Comparison maps}-\ref{improved.comparison2} we shall argue under the assumptions and with the notation of Theorem \ref{main.thm}. In particular $u\in W^{1,p}_\loc(\Omega)$ will always be a local minimiser of the functional \eqref{functional}, where the energy density $f(\cdot)$ satisfies assumptions \eqref{asss.f}, $\omega(\cdot)$ is $1/2$-Dini continuous, and where $\mu\in L(n,1)$; many of the results we are going to prove in the following pages, however, just need  the fact that $\omega$ is solely a concave modulus of continuity. Finally, in Section \ref{more regular sec}, we give the necessary modifications to treat Theorem \ref{second.thm}, where assumptions are slightly stronger.

\subsection{A first comparison estimate}\label{Comparison maps}
For $B_\rho\equiv B_\rho(x_0)\subset B_{2R}(x_0)$, let $v\in u+W^{1,p}_0(B_\rho)$ be the (unique) solution to the following Dirichlet problem:
\begin{equation}\label{funct.frozen}
\left\{
    \begin{array}{c}
   \displaystyle  v \to \min_w
\, \int_{B_\rho}f(x_0,{(u)}_{B_\rho},Dw) \,dx   
\\[10pt]
        w \in u+W^{1,p}_0(B_\rho)\;.
\end{array}\right.
\end{equation}
The rest of this section is devoted to the proof of a comparison estimate between $u$ and $v$, that is \rif{comp.definitive} below. 
Since $h(\xi):=f(x_0,{(u)}_{B_\rho},\xi)$ satisfies \eqref{asss.f}$_{1,2}$, then it also satisfies the growth conditions \eqref{ass.h} and hence $v$ is a solution to the Euler equation \eqref{euler.equation.h}; Lemma \ref{Lemma.V} thus yields 
\begin{multline}\label{int.V.f}
\mean{B_\rho}\big|V_s(Du)-V_s(Dv)\big|^2\dx\\\leq c\mean{B_\rho}
\big[f(x_0,{(u)}_{B_\rho},Du)-f(x_0,{(u)}_{B_\rho},Dv)\big]\dx\;.
\end{multline}
Moreover, by minimality of $v$, using \eqref{asss.f}$_2$
\begin{align}\label{energy.Dv}
\int_{B_\rho}(s+|Dv|\big)^p\dx&\leq \frac1\nu\int_{B_\rho}f(x_0,(u)_{B_\rho},Dv)\dx\nonumber\\
&\leq\frac1\nu \int_{B_\rho}f(x_0,(u)_{B_\rho},Du)\dx\leq \frac L\nu \int_{B_\rho}\big(s+|Du|\big)^p\dx\;.
\end{align}
Since $u$ is a minimiser of \eqref{functional}, we have, rearranging terms
\begin{equation}\label{pre.Sobol.mu}
\int_{B_\rho} \big[f(x,u,Du)-f(x,v,Dv)\big]\dx\leq\int_{B_\rho}(v-u)\mu\dx\;. 
\end{equation}
We proceed by giving an estimate for the right-hand side in the above display; we consider different cases. In the first one we assume that
$p\geq 2$ and $\max\{p,n\}>2$. Using H\"older's and Sobolev's inequalities (for the exponents described below) we have
\begin{align}\label{Sobol.mu+}
 \mean{B_\rho}|u-v||\mu| \dx&\leq \left(\mean{B_\rho}\left|\frac{u-v}{\rho}\right|^{(p')^*}\dx\right)^{1/(p')^*}\rho\left(\mean{B_\rho}|\mu|^q\dx\right)^{1/q}\nonumber\\
 &\leq c\left(\mean{B_\rho}|Du-Dv|^{p'}\dx\right)^{1/p'}\rho\left(\mean{B_\rho}|\mu|^q\dx\right)^{1/q}\;.
\end{align}
Here $(p')^*$ denotes the Sobolev's conjugate of $p'=p/(p-1)$ as follows:
\begin{equation}\label{q}
(p')^*:=\frac{np'}{n-p'}=\frac{np}{np-(n+p)}\qquad\Longrightarrow\qquad q:=[(p')^*]'=\frac{np}{n+p}\;.
\end{equation}
In the case $n=p=2$, on the other hand, we still obtain \eqref{Sobol.mu+} for instance with the choice $q=3/2$. The remaining case is when $1<p<2$; for
\begin{equation}\label{q-}
q:=(p^*)'=\frac{np}{np-(n-p)}
\end{equation}
we have 
\begin{align}\label{Sobol.mu-}
 \mean{B_\rho}|u-v||\mu| \dx&\leq  \left(\mean{B_\rho}\left|\frac{u-v}{\rho}\right|^{p^*}\dx\right)^{1/p^*}\rho\left(\mean{B_\rho}|\mu|^q\dx\right)^{1/q}\nonumber\\
 &\leq c\left(\mean{B_\rho}|Du-Dv|^{p}\dx\right)^{1/p}\rho\left(\mean{B_\rho}|\mu|^q\dx\right)^{1/q}.
\end{align}
Hence, setting $t:=\min\{p',p\}$, in any case we have
\eqn{sti1}
$$
 \mean{B_\rho}|u-v||\mu| \dx \leq c\left(\mean{B_\rho}|Du-Dv|^t\dx\right)^{1/t}\rho\left(\mean{B_\rho}|\mu|^q\dx\right)^{1/q},
$$
with $q$ being defined through \eqref{q}-\eqref{q-} according to the various cases considered, and $c\equiv c(n,p)$. 
The only thing we really mind here is that $q<n$ holds. We now look at \rif{int.V.f} and decompose the integrand in the right-hand side as follows:
\begin{eqnarray}\label{I-V}
&&\nonumber f(x_0,(u)_{B_\rho},Du)-f(x_0,(u)_{B_\rho},Dv)\\ 
&&\nonumber\qquad =\big[f(x_0,(u)_{B_\rho},Du)-f(x,u,Du)\big]\\
&&\nonumber\qquad \qquad  +\big[f(x,u,Du)-f(x,v,Dv)\big]\\
&&\nonumber\qquad \qquad +\big[f(x,v,Dv)-f(x_0,(v)_{B_\rho},Dv)\big]\\
&&\qquad \qquad +\big[f(x_0,(v)_{B_\rho},Dv)-f(x_0,(u)_{B_\rho},Dv)\big]\nonumber \\ && \qquad   =:I+II+III+IV\;.
\label{sti2}
\end{eqnarray}
Using \eqref{asss.f}$_4$ and \eqref{tiome2}, also recalling the definition of $\tilde \omega(\cdot)$ in \rif{tiome}, we start estimating
\begin{align}\label{sti3}
\nonumber \mean{B_\rho}I\dx &\leq \tilde L\,\mean{B_\rho}\big[\omega(\rho)+|u-(u)_{B_\rho}|^\alpha\big]\big(s+|Du|\big)^p\dx\\
&\leq  \tilde L\,\omega(\rho)\mean{B_\rho}\big(s+|Du|\big)^p\dx+\tilde L\,\big[\osc_{B_\rho} u\big]^\alpha\mean{B_\rho}\big(s+|Du|\big)^p\dx \nonumber \\
&\leq  c\tilde \omega(\rho)\mean{B_\rho}\big(s+|Du|\big)^p\dx\;. 
\end{align}
with $c\equiv c(p,\nu, L,\tilde L, \|\mu\|_{L^n})$. 
Similarly we have for the third term
\begin{align}\label{sti4}
\nonumber \mean{B_\rho}III\dx &\leq  \tilde L\,\omega(\rho)\mean{B_\rho}\big(s+|Dv|\big)^p\dx+\tilde L\,\big[\osc_{B_\rho} v\big]^\alpha\mean{B_\rho}\big(s+|Dv|\big)^p\dx\\
&\leq  \tilde L\,\omega(\rho)\mean{B_\rho}\big(s+|Dv|\big)^p\dx+\tilde L\,\big[\osc_{B_\rho} u\big]^\alpha\mean{B_\rho}\big(s+|Dv|\big)^p\dx \nonumber \\
&\leq  c\tilde \omega(\rho)\mean{B_\rho}\big(s+|Du|\big)^p\dx\;. 
\end{align}
and we have used that $\osc_{B_\rho}v\leq \osc_{B_\rho}u$, which is a basic consequence of the maximum principle 
\cite{GT}; we also used \eqref{energy.Dv}. As for the maximum principle, we are using here that $v$ solves the Euler-Lagrange equation $$\divo\, \partial f (x_0,(u)_{B_\rho}, Dv)=0\;.$$ Again we have 
\begin{align}\label{sti5}
\nonumber \mean{B_\rho}IV\dx&\leq \tilde L\mean{B_\rho}|(u)_{B_\rho}-(v)_{B_\rho}|^\alpha\big(s+|Dv|\big)^p\dx
\\ &\leq  c\tilde \omega(\rho)\mean{B_\rho}\big(s+|Du|\big)^p\dx
\end{align}
since, again by maximum principle we have
\begin{equation}\label{max.osc}
\inf_{B_\rho}u\leq v\leq\sup_{B_\rho}u\quad\text{in $B_\rho$}\qquad
 \Longrightarrow\qquad \|u-v\|_{L^\infty(B_\rho)}\leq \osc _{B_\rho} u\;. 
\end{equation}
Finally, using \rif{pre.Sobol.mu} and \rif{sti1} we have 
\eqn{sti6}
$$
\mean{B_\rho}II\dx\leq c\left(\mean{B_\rho}|Du-Dv|^t\dx\right)^{1/t}\rho\left(\mean{B_\rho}|\mu|^q\dx\right)^{1/q}\;.
$$
Connecting this to \rif{int.V.f} and in turn using estimates to \rif{sti2}-\rif{sti5} yields the comparison estimate we were looking for, this is in the following
\begin{Lemma}
Let $u\in W^{1,p}_\loc(\Omega)$ be a local minimiser of the functional \eqref{functional}, where the energy density $f(\cdot)$ satisfies assumptions \eqref{asss.f} and where $\mu\in L(n,1)$. Let $v\in u+W^{1,p}_0(B_\rho)$ be the solution to \trif{funct.frozen}. Then the estimate
\begin{multline}\label{comp.definitive}
\hspace{-3mm}\mean{B_\rho} \big|V_s(Du)-V_s(Dv)\big|^2\dx
\leq c\tilde \omega(\rho)\mean{B_{\rho}}\big(s+|Du|\big)^p\dx\\
+c\left(\mean{B_\rho}|Du-Dv|^t\dx\right)^{1/t}\rho\left(\mean{B_\rho}|\mu|^q\dx\right)^{1/q},
\end{multline}
holds for a constant depending on $n,p,\nu, L,\tilde L$ and $\|\mu\|_{L^n}$, and where
\eqn{fissati}
$$
t:=\min\{p,p'\}\qquad \mbox{and}\qquad 1<q\equiv q(n,p)<n\;.
$$
\end{Lemma}
\subsection{A reverse inequality for minima} Here we prove the following reverse inequality for minimisers of the functional $\F$ defined in \eqref{functional}. For this we need less assumptions than those reported in Theorem \ref{main.thm}. 
\begin{Lemma}\label{reverse}
Let $u$ be a minimiser of \eqref{functional} under assumptions \trif{asss.f}$_2$ and $\mu \in L^n(\Omega)$. Then there exists a constant $c\equiv c (n,p,\nu, L)$ such that the following inequality holds whenever $B_{2\rho}\subset \Omega$ is ball:
\begin{eqnarray}
\nonumber \left( \mean{B_{\rho}} |Du|^p \dx\right)^{1/p} & \leq  & c\left(\mean{B_{2\rho}}(|Du|+s)^{p'}\dx\right)^{1/p'} \\ && \qquad +c\left(\rho^q \mean{
B_{2\rho}}|\mu|^q\dx\right)^{1/[q(p-1)]} \label{riversa}
\end{eqnarray}
where the exponent $q$ is as in \trif{fissati}. 
\end{Lemma}
\begin{proof} We can assume that $p>2$ otherwise $p\leq 2$ implies $p'\geq p$ and the statement is obviously verified; actually Lemma \ref{reverse} will be used only in the case $p>2$. The rest of the proof goes in three different steps. 

{\em Step 1: A preliminary estimate}. We start following the proof of \cite[Theorem 6.5]{G}. We take concentric balls 
$B_{\rho}\subset B_{\varrho_0} \subset B_{\varrho_1}\subset B_{2\rho}$ and a related cut-off function $\eta \in C^{\infty}_0(B_{(\varrho_0+\varrho_1)/2})$ such that
$\eta \equiv 1$ on $B_{\varrho_0}$ and $|D\eta| \leq c/(\varrho_1-\varrho_0)$. We then test the minimality of $u$ using the competitor 
$v:=u-\eta(u-(u)_{B_{2\rho}})$; using the growth conditions on $f(\cdot)$, we then have the following inequality:
\begin{eqnarray*}
 \int_{B_{\varrho_0}} |Du|^p \dx & \leq  &\tilde c \int_{B_{\varrho_1}\setminus B_{\varrho_0}} |Du|^p \dx 
 +  c\int_{B_{\varrho_1}}\left|\frac{u-(u)_{B_{2\rho}}}{\varrho_1-\varrho_0}\right|^p \dx \\&& \quad +\ c
\int_{B_{\varrho_1}} |u-(u)_{B_{2\rho}}||\mu| \dx+c s^p\varrho_1^n\;.
\end{eqnarray*}
By ``filling the hole", that is by adding to both sides of the previous inequality the integral
\[
\tilde c\int_{B_{\varrho_0}} |Du|^p \dx
\]
we come to 
\begin{eqnarray*}
 \int_{B_{\varrho_0}} |Du|^p \dx & \leq  &\theta \int_{B_{\varrho_1}
} |Du|^p \dx +\frac{c}{(\varrho_1-\varrho_0)^p} \int_{B_{\varrho_1}} |u-(u)_{B_{2\rho}}|^p \dx \\&& \qquad +c
\int_{B_{\varrho_1}} |u-(u)_{B_{2\rho}}||\mu| \dx +c s^p\rho^n
\end{eqnarray*}
for $\theta \equiv \theta (n,p,\nu, L)=\tilde c/(\tilde c+1) \in (0,1)$ and again $c \equiv c (n,p,\nu, L)\geq 1$. We can 
therefore apply the iteration Lemma \ref{iterazione} below in order to obtain
\[
 \mean{B_{\rho}} |Du|^p \dx  \leq  c \mean{B_{2\rho}}\left|\frac{u-(u)_{B_{2\rho}}}{2\rho}\right|^p \dx+c
\mean{B_{2\rho}} |u-(u)_{B_{2\rho}}||\mu| \dx+cs^p\;. 
\]
We now estimate the two integrals appearing on the right-hand side of the above inequality. The last term can be estimated exactly as the integral appearing on the right-hand side of \rif{pre.Sobol.mu}. Therefore, proceeding as in \rif{Sobol.mu+}-\rif{sti1}, and with the notation wooed there about $q$, we have 
\begin{eqnarray*}
 \mean{B_{2\rho}}|u-(u)_{B_{2\rho}}||\mu| \dx   & \leq  &c\left(\mean{B_{2\rho}}|Du|^{p'}\dx\right)^{1/p'}\rho\left(\mean{B_{2\rho}}|\mu|^q\dx\right)^{1/q}\\& \leq  &c\left(\mean{B_{2\rho}}|Du|^{p'}\dx\right)^{p/p'}+c\rho^{p'}\left(\mean{B_\rho}|\mu|^q\dx\right)^{p'/q}\;.
\end{eqnarray*}
Moreover, using Sobolev embedding theorem we find  
\[\mean{B_{2\rho}} \left|\frac{u-(u)_{B_{2\rho}}}{2\rho}\right|^p \dx \leq c\left( \mean{B_{2\rho}} |Du|^{p_*} \dx\right)^{p/p_*}
\]
with, as usual, $p_*=np/(n+p)$. 
Connecting the content of the last three displays yields 
\begin{eqnarray}
\nonumber \left( \mean{B_{\rho}} |Du|^p \dx\right)^{1/p} &\leq & c\left( \mean{B_{2\rho}} |Du|^{p_*} \dx\right)^{1/p_*}
+c\left(\mean{B_{2\rho}}|Du|^{p'}\dx\right)^{1/p'} \\ 
&&\quad + c\left(\rho^q \mean{B_{2\rho}}|\mu|^q\dx\right)^{1/[q(p-1)]}+cs\;.\label{riv3}
\end{eqnarray}
The previous inequality actually holds for any ball $B_{2\rho} \subset \Omega$. We now distinguish two cases. 
The first one is when $p_* \leq p'$, and in this case we have finished since \rif{riversa} follows immediately from the inequality in the above display and H\"older's inequality. The other case is when $p'<p_*$ and in order to deal with it we have to use another interpolation argument. This needs a preliminary scaling procedure and this facts are developed in the next two steps. 

{\em Step 2: Rescaling}. Here we recall a standard rescaling procedure. Indeed, for a ball $B_{2\rho } \equiv B_{2\rho }(x_0)\subset \Omega$, if we define the rescaled functions
\eqn{defisca1}
$$
\tilde u(x):= \frac{u(x_0+\rho x)}{\rho }\;, \qquad \qquad \tilde \mu(x):=\rho  \mu(x_0+\rho x)
$$
for $x \in B_1$ and the integrand
\eqn{defisca2}
$$
\tilde f(x, v, \xi) := f(x_0+\rho x, \rho v, \xi)
$$
for 
$ (x, v, \xi)\in B_2 \times \er\times \ern$, it is not difficult to prove that $\tilde u$ is a local minimiser of the functional
\[
w \mapsto \int_{B_2} \tilde f(x,w,Dw)\dx+\int_{B_2} w\tilde\mu\dx\;.
\]
This functional satisfies the same assumptions of the original one considered in \rif{functional}. We will then prove the inequality 
\begin{eqnarray}
\left( \mean{B_1} |D\tilde u|^p \dx\right)^{1/p}& \leq &  c\left(\mean{B_2}(|D\tilde u|+s)^{p'}\dx\right)^{1/p'}\nonumber  \\ && \qquad +c\left( \mean{
B_2}|\tilde \mu|^q\dx\right)^{1/[q(p-1)]} \label{riversas}
\end{eqnarray}
eventually recovering \rif{riversa} by scaling back to $u$. From now on we can therefore reduce to prove \rif{riversas} and this will be done in the third and final step. 

{\em Step 3: Proof of \trif{riversas}.} Inequality \rif{riv3} can be rewritten as 
\begin{eqnarray}
\nonumber \int_{B_{\rho}} |D\tilde u|^p \dx &\leq & \frac{c}{\rho^{p}}\left( \int_{B_{2\rho}} |D\tilde u|^{p_*} \dx\right)^{p/p_*}
+\frac{c}{\rho^{n(p-2)}}\left(\int_{B_{2\rho}}|D\tilde u|^{p'}\dx\right)^{p/p'} \\&& \qquad
+ \frac{c}{\rho^{n(p'/q-1)-p'}}\left( \int_{B_{2\rho}}|\tilde \mu|^q\dx\right)^{p'/q}+cs^p\rho^n\label{ppp}
\end{eqnarray}
that holds whenever $B_{2\rho} \subset B_2$. 
We now consider again concentric balls $B_1\subset B_{\varrho_0} \subset B_{\varrho_1}\subset B_2$ and we set $\rho := (\varrho_1-\varrho_0)/4$ and take a covering of $B_{\varrho_0}$ with a family of balls 
$\{B_{\rho}(y_i)\}_{i\in \{1, \ldots, H\}}$ made of at most $H\approx c(n)\rho^{-n}$ balls, such that 
$y_i\in B_{\varrho_0}$ for every $i\in \{1, \ldots, H\}$. Here $c(n)$ depends only on $n$. Moreover, the covering can be taken in such a way that 
each (doubled) ball $B_{2\rho}(y_i)$ touches at most $8^n$ other similar (doubled) balls from the same family (finite intersection property). Summing up therefore yields
\begin{eqnarray}
&&\nonumber  \int_{B_{\varrho_0}} |D\tilde u|^p \dx \leq  \sum_{i} \int_{B_{\rho}(y_i)} |D\tilde u|^p \dx\\ &&\nonumber 
\qquad \leq  \frac{c}{\rho^{p}}\sum_{i}\left(\int_{B_{2\rho}(y_i)}|D\tilde u|^{p_*} \dx\right)^{p/p_*}
+\frac{c}{\rho^{n(p-2)}}\sum_{i}\left(\int_{B_{2\rho}(y_i)}|D\tilde u|^{p'}\dx\right)^{p/p'} \\&& \qquad
\qquad+ \frac{c}{\rho^{n(p'/q-1)-p'}}\sum_{i}\left( \int_{B_{2\rho}(y_i)}|\tilde \mu|^q\dx\right)^{p'/q}+cHs^p\rho^n \nonumber 
\\ && \nonumber \qquad \leq  \frac{c}{\rho^{p}}\left(\int_{B_{\varrho_1}}|D\tilde u|^{p_*} \dx\right)^{p/p_*}
+\frac{c}{\rho^{n(p-2)}}\left(\int_{B_{\varrho_1}}|D\tilde u|^{p'}\dx\right)^{p/p'} \\&& 
\qquad \qquad+ {c}{\rho^{p'}}\left( \int_{B_{\varrho_1}}|\tilde \mu|^q\dx\right)^{p'/q}+cs^p\;.\label{insert}
\end{eqnarray}
Notice that to perform the estimation for the last sum in the above display we have made use of the elementary Lemma \ref{trivialem} below, since we are in the situation where  $p'<p_*=q$. Moreover, we can use the following interpolation inequality:
\begin{equation}
\label{eqn:interp}
{ \|D\tilde u\|}_{L^{p_*}(B_{\varrho_1})} \leq  \|D\tilde u\|_{L^{p'}(B_{\varrho_1})}^{\theta}  \|D\tilde u\|_{L^{p}(B_{\varrho_1})}^{1-\theta}\;,
\end{equation}
that holds for $\theta \in (0,1)$ such that
\eqn{theta}
$$
\frac{1}{p_*} = \frac{\theta}{p'}+\frac{1-\theta}{p}\;.$$
Inserting \rif{eqn:interp} in \rif{insert} yields 
\begin{eqnarray*}
\nonumber \int_{B_{\varrho_0}} |D\tilde u|^p \dx &\leq & \frac{c}{\rho^{p}}\left( \int_{B_{\varrho_1}} |D\tilde u|^{p'} \dx\right)^{\theta p/p'}
\left( \int_{B_{\varrho_1}} |D\tilde u|^{p} \dx\right)^{1-\theta}\\ && \qquad
+\frac{c}{\rho^{n(p-2)}}\left(\int_{B_{\varrho_1}}|D\tilde u|^{p'}\dx\right)^{p/p'} \\&& \qquad
+ {c}{\rho^{p'}}\left( \int_{B_{\varrho_1}}|\tilde \mu|^q\dx\right)^{p'/q}+cs^p\;.
\end{eqnarray*}
By using Young's inequality and recovering the full notation we then find, after a few elementary manipulations
\begin{multline*}
 \left(\int_{B_{\varrho_0}} |D\tilde u|^p \dx\right)^{1/p} \leq \frac{1}{2}\left(\int_{B_{\varrho_1}} |D\tilde u|^p \dx\right)^{1/p}
\\
+\frac{c}{(\varrho_1-\varrho_0)^{n(p-2)/p}}\left(\int_{B_{\varrho_1}}|D\tilde u|^{p'}\dx\right)^{1/p'}+ c\left( \int_{B_{\varrho_1}}|\tilde \mu|^q\dx\right)^{1/[q(p-1)]}+cs.
\end{multline*}
The latter inequality holds whenever $1\leq \varrho_0< \varrho_1 \leq 2$ for a constant $c\equiv c (n,p,\nu, L)$. At this point \rif{riversas} follows applying Lemma \ref{iterazione} below with $\rho_0=1$ and $\rho_1=2$. 
\end{proof}
\begin{Lemma} \label{iterazione}
Let $h:[\rho_0,\rho_1] \to \mathbb R $ be a nonnegative and bounded function, and let $\theta \in (0,1)$ and 
$A,B\geq 0$, $p>0$ be numbers. Assume that
\[ h(\varrho_0) \leq\theta h(\varrho_1)+ \frac{A}{(\varrho_1-\varrho_0)^{\gamma_1}}+
 B
\]
holds for every choice of $\varrho_0$ and $\varrho_1$ such that 
 $\rho_0 \leq
\varrho_0< \varrho_1\leq \rho_{1}$. Then the following inequality holds with $c\equiv c(\theta , \gamma_1, \gamma_2 )$:
\[ h(\rho_0) \leq  \frac{cA}{(\rho_1-\rho_0)^{\gamma_1}}
+cB\;.
\]
\end{Lemma}
The next lemma is a consequence of the concavity of the function $t \to t^{\gamma}$ for $\gamma \leq 1$. 
\begin{Lemma}\label{trivialem} Let $\{a_i\}_{1\leq i \leq H}$ be non-negative numbers and $\gamma\leq 1$, then the following inequality holds:
\eqn{simpleine}
$$
\sum_{i=1}^H a_i^{\gamma} \leq H^{1-\gamma} \left(\sum_{i=1}^H a_i\right)^{\gamma}\;.
$$
\end{Lemma}
\subsection{A second comparison  in the degenerate case} \label{improved.comparison}
Again with $B_{2R}\equiv B_{2R}(x_0)\subset\Omega$ fixed as described at the beginning of Section \ref{beg sec}, that is with $2R \leq R_H\leq 1$, we consider a sequence of shrinking balls 
\eqn{setting}
$$\{B_j\}_{j\in\N_0},\quad B_0=B_R,\quad B_{j+1}=\delta B_j,\quad R_j :=\delta^jR, \quad j\in\N_0\;,$$
 for some $\delta\in(0,1/8)$ which we shall also fix later on; clearly $B_j=B_{\delta^j R}=B_{R_j}$. 
 We shall moreover denote
 $$
 \tb := \frac{1}{2}B_j\;.
 $$
 Notice the inclusions 
 $$
 \ldots \tilde {B_j} \subset B_j \subset \tilde B_{j-1} \subset  B_{j-1} \ldots\subset B_{R}
 $$
 that follow since $\delta \leq 1/4$. 
 Accordingly, for every integer $j \geq 0$ we define $v_j\in u+W^{1,p}_0(\tilde B_j)$ be the solution to the following Dirichlet problem:
\begin{equation}\label{funct.frozenj}
\left\{
    \begin{array}{c}
   \displaystyle  v_j \to \min_w
\, \int_{\tb}f(x_0,{(u)}_{\tb},Dw) \,dx   
\\[10pt]
        w \in u+W^{1,p}_0(\tilde B_j)\;.
\end{array}\right.
\end{equation}
These are problems of the type \rif{funct.frozen}. 
We start from the case $p\geq2$; in this case we have the following estimate:
\begin{Lemma}\label{Lemma:comparison.I1} Let $u\in W^{1,p}_\loc(\Omega)$ be a local minimiser of the functional \eqref{functional}, where the energy density $f(\cdot)$ satisfies assumptions \eqref{asss.f} with $p \geq 2$ and where $\mu\in L(n,1)$. Let $v_j\in u+W^{1,p}_0(\tilde B_j)$ be the solution to \trif{funct.frozenj}. 
Suppose that, for some positive $j\in\N$ and some $\lambda\geq1$, there holds
\begin{equation}\label{ass.mean0}
\left(\mean{B_{j-1}}\big(s+|Du|\big)^{p'}\dx\right)^{1/p'}+\left(\mean{B_{j}}\big(s+|Du|\big)^{p'}\dx\right)^{1/p'}\leq
 \lambda
\end{equation}
together with
\begin{equation}\label{ass.meas}
R_{j}\left(\mean{B_{j}}|\mu|^q\dx\right)^{1/q}\leq \Big(\frac\lambda A\Big)^{p-1},
\quad R_{j-1}\left(\mean{B_{j-1}}|\mu|^q\dx\right)^{1/q}\leq \Big(\frac\lambda A\Big)^{p-1}
\end{equation} 
for some $A\geq1$, where $1<q<n$ comes from \trif{fissati}. Then
\begin{equation}\label{compa.k}
\mean{\tilde B_{k}}|Du-Dv_{k}|^p\dx\leq  \Cl[c]{c:compa.k}\Bigl[\tilde \omega(R_{k})+A^{-p}\Bigr]\lambda^p
\end{equation}
holds for $k=j,j-1$ and a constant $\Cr{c:compa.k}$ depending on $n,p,\nu, L,\tilde L,\|\mu\|_{L^n}$. If moreover
\begin{equation}\label{ass.lambda}
\frac\lambda{B}\leq |Dv_{j-1}|\leq B\lambda \quad\text{in $B_j$}
\end{equation}
holds for some other constant $B\geq1$, then
\begin{multline} \label{compa.I1}
\left(\mean{\tilde B_j}|Du-Dv_j|^{p'}\dx\right)^{1/p'}\\ \leq \Cl[c]{c:compa.I1}
[\tilde \omega(R_j)]^{1/2}\lambda+\Cr{c:compa.I1}\lambda^{2-p}R_{j-1}
\left(\mean{B_{j-1}}|\mu|^q\dx\right)^{1/q}
\end{multline}
holds with $\Cr{c:compa.I1}$ depending on $n,p,\nu, L,\tilde L,\|\mu\|_{L^n},\delta, B$.
\end{Lemma}
\begin{proof} By applying \rif{riversa} with $B_\rho \equiv \tilde B_j, \tilde B_{j-1}$, and using \rif{ass.mean0}-\rif{ass.meas}, we have that 
\begin{equation}\label{ass.mean}
\mean{\tilde B_j}\big(s+|Du|\big)^p\dx+\mean{\tilde B_{j-1}}\big(s+|Du|\big)^p\dx\leq c\lambda^p
\end{equation}
holds for a constant $c\equiv c(n,p,\nu, L)$. 
Matching \eqref{comp.definitive} with \eqref{ass.mean} yields
\begin{multline}\label{comp.p.k}
\mean{\tilde B_k}\big|V_s(Du)-V_s(Dv_k)\big|^2\dx\leq c\,\tilde \omega(R_k)\lambda^p\\
+c\left(\mean{\tilde B_k}|Du-Dv_k|^{p'}\dx\right)^{1/p'} R_k\left(\mean{B_k}|\mu|^q\dx\right)^{1/q}
\end{multline}
for $k=j,j-1$. Moreover, using \eqref{equiv.V.p} to estimate from below the left-hand side in the above display, and also using H\"older's and then Young's inequality, we have
\begin{align*}
\mean{\tilde B_{k}}&|Du-Dv_{k}|^p\dx \leq c\,\tilde \omega(R_k)\lambda^p\\
&\hspace{25mm}+c\left(\mean{\tilde B_{k}}|Du-Dv_{k}|^p\dx\right)^{1/p} R_k\left(\mean{B_{k}}|\mu|^q\dx\right)^{1/q}\\
&\leq c\tilde \omega(R_k)\lambda^p+c\,R_{k}^{p'}\left(\mean{B_{k}}|\mu|^q\dx\right)^{p'/q}+\frac12\mean{\tilde B_{k}}|Du-Dv_{k}|^p\dx
\end{align*}
so that reabsorbing the last integral on the right-hand side we have that 
\eqn{intermedia}
$$
\mean{\tilde B_{k}}|Du-Dv_{k}|^p\dx \leq c\,\tilde \omega(R_k)\lambda^p+cR_{k}^{p'}\left(\mean{B_{k}}|\mu|^q\dx\right)^{p'/q}
$$
and using \eqref{ass.meas} we deduce \eqref{compa.k}. We now proceed with the proof of \rif{compa.I1}. Since both $\omega(R_k)$ and $R_k$ are smaller than one and $A\geq 1$, recalling \rif{tiome}, from \rif{compa.k} it also follows
\begin{equation*}
\mean{\tilde B_{k}}|Du-Dv_{k}|^p\dx\leq c\lambda^p
\end{equation*}
for $k=j,j-1$ and also
\begin{align}\label{double.compa.k}
 \mean{\tilde B_j}|Dv_j-Dv_{j-1}|^p\dx& \leq 2^{p-1}\mean{\tilde B_j}|Du-Dv_j|^p\dx\notag\\
 &\qquad\qquad+2^{p-1}\delta^{-n}\mean{\tilde B_{j-1}}|Du-Dv_{j-1}|^p\dx\notag\\
 &\leq c\,\tilde \omega(R_{j-1})\lambda^p+cR_{j-1}^{p'}\left(\mean{B_{j-1}}|\mu|^q\dx\right)^{p'/q}\notag\\
 &\leq c\lambda^p
\end{align}
with $c$, in both cases, depending on $n,p,\nu, L,\tilde L,\|\mu\|_{L^n},\delta$. Now we use 
\eqref{ass.lambda} to infer
\begin{equation}\label{first.p'}
\mean{\tilde B_j}|Du-Dv_j|^{p'}\dx\leq c\lambda^{p'(2-p)}\mean{\tilde B_j}|Dv_{j-1}|^{p'(p-2)}|Du-Dv_j|^{p'}\dx
\end{equation}
with $c$ depending on $p,B$. We split
\begin{multline}\label{split.p'}
\mean{\tilde B_j}|Dv_{j-1}|^{p'(p-2)}|Du-Dv_j|^{p'}\dx\leq c\mean{\tilde B_j}|Dv_j|^{p'(p-2)}|Du-Dv_j|^{p'}\dx\\
+c\mean{\tilde B_j}|Dv_{j-1}-Dv_j|^{p'(p-2)}|Du-Dv_j|^{p'}\dx=:V+VI
\end{multline}
and we estimate the two terms separately, starting from the second one. Using Young's inequality with exponents $p/p'=p-1$ and $(p-1)/(p-2)$ (only when $p>2$), and recalling \rif{intermedia} and \eqref{double.compa.k}, we estimate
\begin{align*}
VI&\leq c\mean{\tilde B_j}|Du-Dv_j|^p\dx+c\mean{\tilde B_j}|Dv_j-Dv_{j-1}|^{p}\dx\nonumber\\
&\leq c\tilde \omega(R_{j-1})\lambda^p+ cR_{j-1}^{p'}
\left(\mean{B_{j-1}}|\mu|^q\dx\right)^{p'/q}
\end{align*}
so that, for a constant depending on $n,p,\nu, L,\tilde L,\|\mu\|_{L^n},\delta$ it holds that 
\begin{equation}\label{II}
 \hspace{-3mm}\lambda^{p'(2-p)}VI\leq c\tilde\omega(R_{j-1})\lambda^{p'}+ c \lambda^{p'(2-p)}R_{j-1}^{p'}\left(\mean{B_{j-1}}|\mu|^q\dx\right)^{p'/q}\;.
\end{equation}
We now estimate $\lambda^{p'(2-p)}V$. To this aim we preliminary note that the estimate
\eqn{size.is.enough}
$$
 \mean{\tilde B_j}|Dv_j|^p\dx\leq c\mean{\tilde B_{j-1}}|Dv_j-Dv_{j-1}|^p\dx+c\mean{\tilde B_j}|Dv_{j-1}|^p\dx\leq c\lambda^p
$$
holds by \eqref{ass.lambda} and \eqref{double.compa.k}. Hence we have, using again H\"older's inequality (when $p>2$) with conjugate exponents $(2/p',2(p-1)/(p-2))$ and \eqref{equiv.V}
\begin{align}\label{I}
\hspace{-10mm}\lambda^{p'(2-p)}V&=\lambda^{p'(2-p)}\mean{\tilde B_j}|Dv_j|^{p'(p-2)/2+p'(p-2)/2}|Du-Dv_j|^{p'}\dx\nonumber\\
&\leq \lambda^{p'(2-p)}\left(\mean{\tilde B_j}|Dv_j|^{p-2}|Du-Dv_j|^2\dx\right)^{p'/2}\left(\mean{\tilde B_j}|Dv_j|^p\dx\right)^{(2-p')/2}\nonumber\\
&\leq c\lambda^{p'(2-p)/2}\left(\mean{\tilde B_j}\big(s+|Du|+|Dv_j|\big)^{p-2}|Du-Dv_j|^2\dx\right)^{p'/2}\nonumber\\
&\leq  c\lambda^{p'(2-p)/2}\left(\mean{\tilde B_j}\big|V_s(Du)-V_s(Dv_j)\big|^2\dx\right)^{p'/2}\;.
\end{align}
In turn, by using \rif{comp.p.k}, we estimate
\begin{align*}
&\lambda^{p'(2-p)/2}\left(\mean{\tilde B_j}\big|V_s(Du)-V_s(Dv_j)\big|^2\dx\right)^{p'/2}\nonumber\\
&\leq c[\tilde \omega(R_j)]^{p'/2}\lambda^{p'(2-p)/2+pp'/2}\nonumber\\
&\hspace{1cm}+c\lambda^{p'(2-p)/2}\left(\mean{\tilde B_j}|Du-Dv_j|^{p'}\dx\right)^{1/2} \left(R_j^q\mean{B_j}|\mu|^q\dx\right)^{p'/(2q)} \nonumber\\
&\leq c[\tilde \omega(R_j)]^{p'/2}\lambda^{p'}+c_{\eps}\lambda^{p'(2-p)}R_j^{p'}\left(\mean{B_j}|\mu|^q\dx\right)^{p'/q}\nonumber\\
&\hspace{1cm}+\varepsilon\mean{\tilde B_j}|Du-Dv_j|^{p'}\dx\;,
\end{align*}
$c_{\eps}$ depending on $n,p,\nu, L,\tilde L,\|\mu\|_{L^n},\delta$ and on $\eps$. Merging all the estimate found from display \rif{first.p'} on and making a few elementary manipulations yields, with $\eps \in (0,1)$,
\begin{multline*}
\mean{\tilde B_j}|Du-Dv_j|^{p'}\dx\leq c[\tilde \omega(R_{j-1})]^{p'/2}\lambda^{p'}\\
+c_{\eps}\lambda^{p'(2-p)}R_{j-1}^{p'}\left(\mean{B_{j-1}}|\mu|^q\dx\right)^{p'/q} +c\,\eps\mean{\tilde B_j}|Du-Dv_j|^{p'}\dx
\end{multline*}
so that \eqref{compa.I1} follows by choosing $\eps$ small enough in order to reabsorb the last integral on the left-hand side.
\end{proof}

\subsection{A second comparison estimate in the singular case} \label{improved.comparison2}
Here we derive a suitable analog to \rif{compa.I1} in the so-called singular case $1<p<2$. The general setting remains the one fixed in \rif{setting}-\rif{funct.frozenj}. 
\begin{Lemma}\label{comparison.I1-} Let $u\in W^{1,p}_\loc(\Omega)$ be a local minimiser of the functional \eqref{functional}, where the energy density $f(\cdot)$ satisfies assumptions \eqref{asss.f} with $1<p<2$ and where $\mu\in L(n,1)$. Let $v_j\in u+W^{1,p}_0(\tilde B_j)$ be the solution to \trif{funct.frozenj}. 
Suppose that
\begin{equation}\label{ass.singular}
\mean{B_j}\big(s+|Du|\big)^p\dx\leq \lambda^p,\qquad R_{j}\left(\mean{B_{j}}|\mu|^q\dx\right)^{1/q}\leq \lambda^{p-1}
\end{equation}
hold for some $j\in\N_0$ and $\lambda\geq1$, where $q$ has been fixed in \trif{fissati}. Then the inequality
\eqn{compa.I1-}
$$
\left(\mean{\tilde B_j}|Du-Dv_j|^p\dx\right)^{1/p}\leq \Cl[c]{c:compa.I1-}[\tilde \omega(R_j)]^{1/2}\lambda+\Cr{c:compa.I1-}\lambda^{2-p}R_j\left(\mean{B_j}|\mu|^q\dx\right)^{1/q}
$$
holds for a constant $\Cr{c:compa.I1-}$ depending on $n,p,\nu, L,\tilde L,\|\mu\|_{L^n}$.
\end{Lemma}
\begin{proof}
By \eqref{equiv.V} we have
\[
|Du-Dv_j|^p\leq c|V_s(Du)-V_s(Dv_j)|^p\big(s+|Du|+|Dv_j|\big)^{p(2-p)/2}
\]
and using H\"older's and triangle's inequalities
\begin{multline*}
\mean{\tilde B_j}|Du-Dv_j|^p\dx\leq c\left(\mean{\tilde B_j}|V_s(Du)-V_s(Dv_j)|^2\dx\right)^{p/2}\times\\
\times\left[\left(\mean{\tilde B_j}\big(s+|Du|\big)^p\dx\right)^{(2-p)/2}+\left(\mean{\tilde B_j}|Du-Dv_j|^p\dx\right)^{(2-p)/2}\right].
\end{multline*}
Then we use \rif{comp.definitive} and \eqref{ass.singular}$_1$ to estimate
\begin{align*}
&\hspace{-4mm}\left(\mean{\tilde B_j}|V_s(Du)-V_s(Dv_j)|^2\dx\right)^{p/2}\left(\mean{\tilde B_j}\big(s+|Du|\big)^p\dx\right)^{(2-p)/2}\\
&\leq c\lambda^{p(2-p)/2}\left(\mean{\tilde B_j}|V_s(Du)-V_s(Dv_j)|^2\dx\right)^{p/2} \\
& \leq c[\tilde \omega(R_j)]^{p/2}\lambda^{p(2-p)/2+p^2/2}\\
&\qquad+c\lambda^{p(2-p)/2}\left(\mean{\tilde B_j}|Du-Dv_j|^p\dx\right)^{1/2} \left(R_j^q\mean{B_j}|\mu|^q\dx\right)^{p/(2q)}\\
& \leq c[\tilde \omega(R_j)]^{p/2}\lambda^p+c\lambda^{p(2-p)}R_j^p\left(\mean{B_j}|\mu|^q\dx\right)^{p/q}\\
&\qquad+\frac14\mean{\tilde B_j}|Du-Dv_j|^p\dx\;.
\end{align*}
Similarly, we have
\begin{align*}
&\hspace{-6mm}\left(\mean{\tilde B_j}|V_s(Du)-V_s(Dv_j)|^2\dx\right)^{p/2}\left(\mean{\tilde B_j}|Du-Dv_j|^p\dx\right)^{(2-p)/2}\\
&\leq c\mean{\tilde B_j}|V_s(Du)-V_s(Dv_j)|^2\dx +\frac18\mean{\tilde B_j}|Du-Dv_j|^p\dx\\
& \leq  c\tilde \omega(R_j)\lambda^p+c\left(\mean{\tilde B_j}|Du-Dv_j|^p\dx\right)^{1/p}R_j\left(\mean{B_j}|\mu|^q\dx\right)^{1/q}
\\&\hspace{7cm}+\frac18\mean{\tilde B_j}|Du-Dv_j|^p\dx
\\
& \leq  c\tilde\omega(R_j)\lambda^p
+cR_j^{p'}\left(\mean{B_j}|\mu|^q\dx\right)^{p'/q}+\frac14\mean{\tilde B_j}|Du-Dv_j|^p\dx\;.
\end{align*}
We then conclude reabsorbing the last integral on the left 
hand side and finally noticing that, using \eqref{ass.singular}$_2$, we have
\begin{eqnarray*}
R_j^{p'}\left(\mean{B_j}|\mu|^q\dx\right)^{p'/q}&=&\left(R_j^{q}\mean{B_j}|\mu|^q\dx\right)^{p/q+p(2-p)/[q(p-1)]}
\\ &\leq &\lambda^{p(2-p)}\left(R_j^q\mean{B_j}|\mu|^q\dx\right)^{p/q}.
\end{eqnarray*}


\end{proof}

\subsection{More regular integrands}\label{more regular sec} Here we state and prove the versions of Lemmas \ref{Lemma:comparison.I1} and \ref{comparison.I1-} which are necessary to prove Theorem \ref{second.thm}. Therefore in this section we consider minimisers $u\in W^{1,p}_\loc(\Omega)$ of the functional \eqref{functional} assuming that $f(\cdot)$ satisfies \eqref{asss.f} and also \eqref{further.asss.f}, where $\omega(\cdot)$ is Dini continuous; 
as usual we take $\mu\in L(n,1)$. The modifications essentially occur in the estimates contained in Section \ref{Comparison maps}; we introduce
\[
\bar \omega(\rho):=\omega(\varrho) + \varrho^{\alpha\gamma/2}\,,
\]
as in \eqref{tiome}; if $\omega(\cdot)$ is Dini continuous, the same holds for $\bar \omega(\cdot)$. Let then $\Omega '\Subset \Omega$ be an open subset. Again we have
\[
\big[\osc_{B_\rho} u\big]^\alpha\leq c \tilde \omega(\varrho)
\]

We restart from \rif{int.V.f}, similarly as in \eqref{I-V}; using again the minimality of $u$ in \eqref{pre.Sobol.mu} we have
\begin{align*}
\mean{B_\rho}&\big|V_s(Du)-V_s(Dv)\big|^2\dx\\
&\leq  c\mean{B_\rho}\big[f(x_0,(u)_{B_\rho},Du)-f(x_0,(u)_{B_\rho},Dv)\big]\dx\\
&\leq  c\mean{B_\rho}\big[f(x_0,(u)_{B_\rho},Du)-f(x_0,(u)_{B_\rho},Dv)\big]\dx\\
&\qquad+  c\mean{B_\rho}\big[f(x,v,Dv)-f(x,u,Du)\big]\dx+  c\mean{B_\rho}(v-u)\mu\dx\\
&=  c\mean{B_\rho}\big[f(x_0,(u)_{B_\rho},Du)-f(x,u,Du)\big]\dx\\
&\qquad-  c\mean{B_\rho}\big[f(x_0,(u)_{B_\rho},Dv)-f(x,u,Dv)\big]\dx\\
&\qquad+  c\mean{B_\rho}\big[f(x,v,Dv)-f(x,u,Dv)\big]\dx+  c\mean{B_\rho}(v-u)\mu\dx\\
&:=VII-VIII+IX+X\;.
\end{align*}
Accordingly, we define
\[
G(x,\xi):=f(x_0,(u)_{B_\rho},\xi)-f(x,u(x),\xi),\qquad x\in \Omega, \xi\in\R^n
\]
and we notice that with this notation, we have
\begin{align*}
VII-VIII&=c\mean{B_\rho}\big[G(x,Du)-G(x,Dv)\big]\dx\\
&=c\mean{B_\rho} \int_0^1\langle \partial G(x,\lambda Du+(1-\lambda)Dv),Du-Dv\rangle\,d\lambda\dx\;.
\end{align*}
Using assumption \eqref{further.asss.f} we estimate
\begin{multline*}
|\partial G(x,\lambda Du+(1-\lambda)Dv)|=|\partial f(x_0,(u)_{B_\rho},\lambda Du+(1-\lambda)Dv)\\
\hspace{5cm}-\partial f(x,u(x),\lambda Du+(1-\lambda)Dv)|\\
\leq \tilde L\big[ \omega(\rho)+|u-(u)_{B_\rho}|^{{\alpha}}\big]\big(s+|\lambda Du+(1-\lambda)Dv|\big)^{p-1}\;.
\end{multline*}
Easy manipulations, using Young's inequality, \rif{tiome2} and noting that $p-1=(p-2)/2+p/2$ give
\begin{align*}
&\hspace{-5mm}VII-VIII\leq  \tilde L
\mean{B_\rho}\big\{\omega(\rho)+\big[\osc_{B_\rho} u\big]^{\alpha}\big\}\big(s+|Du|+|Dv|\big)^{p-1}|Du-Dv|\dx\\
&\hspace{-2mm} \leq  \eps\mean{B_\rho}\big(s+|Du|+|Dv|\big)^{p-2}|Du-Dv|^2\dx\\
& \quad\ + c_{\eps}\left\{[\omega(\rho)]^2+\big[\osc_{B_\rho} u\big]^{2\alpha}\right\}
\mean{B_\rho}\big(s+|Du|+|Dv|\big)^{p}\dx\\
&\hspace{-2mm} \leq  \eps\mean{B_\rho}\big|V_s(Du)-V_s(Dv)\big|^2\dx + c_{\eps}[\bar \omega(\rho)]^2
\mean{B_\rho}\big(s+|Du|\big)^{p}\dx
\end{align*}
where we have also used \rif{equiv.V} and \rif{energy.Dv}. The term $IX$ can be  estimated as in \rif{sti5} and this yields 
\[
IX \leq c[\bar \omega(\rho)]^2\mean{B_\rho}\big(s+|Du|\big)^p\dx
\]
and finally $X$ is estimated exactly as in \eqref{sti1}. Combining the above estimates leads therefore to the proof of the following
\begin{Lemma}
Let $u\in W^{1,p}_\loc(\Omega)$ a local minimiser of the functional \eqref{functional} and suppose that $f(\cdot)$ satisfies \eqref{asss.f} and also \eqref{further.asss.f} and with $\mu\in L(n,1)$. Let $v\in u+W^{1,p}_0(B_\rho)$ be the solution to \trif{funct.frozen}. Then the estimate
\begin{multline}\label{comp.definitive2}
\hspace{-3mm}\mean{B_\rho} \big|V_s(Du)-V_s(Dv)\big|^2\dx
\leq c[\bar \omega(\rho)]^2\mean{B_{\rho}}\big(s+|Du|\big)^p\dx\\
+c\left(\mean{B_\rho}|Du-Dv|^t\dx\right)^{1/t}\rho\left(\mean{B_\rho}|\mu|^q\dx\right)^{1/q}
\end{multline}
holds for a constant depending on $n,p,\nu, L,\tilde L$ and $\|\mu\|_{L^n}$, and where
\eqn{fissati-parte2}
$$
t:=\min\{p,p'\}\qquad \mbox{and}\qquad 1<q<n\;.
$$
\end{Lemma}
With the previous result we can then derive the analogs of Lemmas \ref{Lemma:comparison.I1} and \ref{comparison.I1-}, that we report below in sequence. The proof is identical to the one of Lemmas \ref{Lemma:comparison.I1} and \ref{comparison.I1-} once 
we use \rif{comp.definitive2} instead of \rif{comp.definitive}. 
\begin{Lemma}\label{Lemma:improved.comparison.I1} 
Let $u\in W^{1,p}_\loc(\Omega)$ a local minimiser of the functional \eqref{functional} and suppose that $f(\cdot)$ satisfies \eqref{asss.f} and also \eqref{further.asss.f} with $p\geq2$ and with $\mu\in L(n,1)$. Let $v_j\in u+W^{1,p}_0(\tilde B_j), v_{j-1}\in u+W^{1,p}_0(\tilde B_{j-1})$ be the solutions to \trif{funct.frozenj} and assume that, for some positive $j\in\N$, $A\geq 1$ and some $\lambda\geq1$, there holds
\trif{ass.mean} and \trif{ass.meas}. Then
\begin{equation}\label{compa.k-parte2}
\mean{\tilde B_{k}}|Du-Dv_{k}|^p\dx\leq  \Cl[c]{c:compa.k-parte2}
\left\{[\bar \omega(R_{k})]^2+A^{-p}\right\}\lambda^p
\end{equation}
holds for $k=j,j-1$ and a constant $\Cr{c:compa.k-parte2}$ depending on $n,p,\nu, L,\tilde L,\|\mu\|_{L^n}$. If moreover \trif{ass.lambda}
holds for some other constant $B\geq1$, then
\begin{multline}\label{compa.I1-parte2}
\left(\mean{\tilde B_j}|Du-Dv_j|^{p'}\dx\right)^{1/p'}\leq \Cl[c]{c:compa.I1-parte22}
\bar\omega(R_j)\lambda\\+\Cr{c:compa.I1-parte22}\lambda^{2-p}R_{j-1}
\left(\mean{B_{j-1}}|\mu|^q\dx\right)^{1/q}
\end{multline}
holds with $\Cr{c:compa.I1-parte22}$ depending on $n,p,\nu, L,\tilde L,\|\mu\|_{L^n},\delta, B$.
\end{Lemma}

\begin{Lemma}\label{improvedcomparison.I1-} Let $u\in W^{1,p}_\loc(\Omega)$ a local minimiser of the functional \eqref{functional} and suppose that $f(\cdot)$ satisfies \eqref{asss.f} and also \eqref{further.asss.f} with $1<p<2$ and with $\mu\in L(n,1)$. Let $v_j,v_{j-1}$ be as in Lemma \ref{Lemma:improved.comparison.I1}. Suppose that \trif{ass.singular} holds for some $j\in\N$ and $\lambda\geq1$, where $q$ has been fixed in \trif{fissati-parte2}. Then the inequality
$$
\left(\mean{\tilde B_j}|Du-Dv_j|^p\dx\right)^{1/p}\leq \Cl[c]{c:compa.I1-parte2}\bar \omega(R_j)\lambda+\Cr{c:compa.I1-parte2}\lambda^{2-p}R_j\left(\mean{B_j}|\mu|^q\dx\right)^{1/q}
$$
holds with a constant $\Cr{c:compa.I1-parte2}$ depending on $n,p,\nu, L,\tilde L,\|\mu\|_{L^n}$.
\end{Lemma}
\section{Proof of Theorems \ref{main.thm}-\ref{second.thm}: gradient bounds}
Here we continue the proof of Theorems \ref{main.thm}-\ref{second.thm} deriving local $L^\infty$ bounds for the gradient of 
minima. We shall actually provide a full proof of the fact that $Du$ is locally bounded in the case of 
Theorem \ref{main.thm} when $p \geq 2$. The case $1<p<2$ and the one of Theorem \ref{second.thm} can be then obtained similarly, and we shall provide remarks on how to make the necessary modifications in Section \ref{rem sec} below. Anyway, 
the proof below is written in a way that makes its adaptation to the case $p<2$ easier.  
We therefore start considering the case $p \geq 2$ and fixing the following quantities:
\begin{equation}\label{delta.B}
B=12^n\Cr{sup.est},\qquad  \delta:=\min\{\bar\tau,\bar\delta\}/8\in (0,1/16)\;.
\end{equation}
Here
\begin{itemize}
\item $\Cr{sup.est}$ is the constant appearing in the $\sup$ estimate of Lemma \ref{Lem.regolarity}
\item $\bar\tau$ is the constant appearing in \rif{sigma} from Lemma \ref{Lemmatau} with the choice 
$\tilde A= 3^n\Cr{sup.est}$
\item $\bar\delta$ is provided by Lemma \ref{ex.deca} for the choice $\bar\eps=4^{-10p}$.
\end{itemize}
In this way both $B$ and $\delta$ are determined as functions of the fixed parameters $n,p,\nu, L$. We continue by fixing, according to the choice of $\delta$ in \eqref{delta.B}, the smallest natural number $\tilde \ell\in\N$, larger than three, making the following inequality true:
\begin{equation}\label{choice.ell}
\delta^{(\beta\tilde \ell-n)p'}\leq  \frac1{10^{10np} (\Cr{sup.est}\Cr{osc.est})^p}
\end{equation}
with $\Cr{osc.est}$ and $\beta$ being fixed in Lemma \ref{Lem.regolarity}; this yields a dependence of $\tilde \ell$ only on $n,p,\nu, L$. We then choose $H_1,H_2\geq1$ as follows:
\[
H_1:=\frac{10^{10p}}{\delta^{4n}}+8^3c_1+10^n,\qquad 
H_2:=\frac{10^{10p}(\Cr{c:compa.k}+\Cr{c:compa.I1})}{\delta^{2n(\tilde \ell+4p)}}\;.
\]
We recall that $\Cr{c:compa.k}$ has been defined in \rif{compa.k} and $\Cr{c:compa.I1}$ in \rif{compa.I1}, this last one, for the numbers $B$ and $\delta$ fixed in \rif{delta.B}. All in all, this yields a dependence of $H_1$ and $H_2$ only on $n,p,\nu, L, \tilde L, \|\mu\|_{L^n}$. 
Finally, we can chose the threshold radius $R_0\leq \min\{R_H/2,1/4\}$ such that 
\eqn{choice.R}
$$
\left\{
 \begin{array}{c}
\displaystyle\tilde \omega(R_0)\leq \frac{\delta^{2n(\tilde \ell+4p)p}}{10^{10p^2}\Cr{c:compa.k}}\\ [16 pt]
\displaystyle \delta^{-1}\int_0^{4R_0}[\tilde \omega(\rho)]^{1/2}\frac{d\rho}{\rho}\leq \frac{\delta^{2n}}{10^{10p}\Cr{c:compa.I1}}
\end{array}
\right.
$$
are satisfied. Finally, for $2R\leq R_0$ such that $B_{2R}\subset \Omega$, we define the sequence of radii and balls $R_j$ and $B_j, \tilde B_j$ as in \rif{setting} for the choice of $\delta$ made in \rif{delta.B}. Notice that in this way $R_0$ is a quantity depending only on $n,p,\nu, L, \tilde L, \|\mu\|_{L^n}$ and $\omega(\cdot)$. We are going to prove the following estimate:
\begin{multline}\label{pointwise}
 |Du(x_0)|\leq \lambda:=H_1\left(\mean{B_R(x_0)}\big(s+|Du|\big)^{p'}\dx\right)^{1/p'}\\+H_2\left[\sum_{j=0}^\infty R_j\left(\mean{B_j}|\mu|^q\dx\right)^{1/q}\right]^{1/(p-1)}
\end{multline}
with $q\in(1,n)$ is defined as in \eqref{fissati}. This estimate, via a standard covering argument, will then lead to the local boundedness of $Du$. To this aim, we introduce the quantities:
\begin{equation}\label{mu.recall0}
a_j:=\big|(Du)_{B_j}\big|,\qquad E_j\equiv E_j(x_0):=\left(\mean{B_j}\big|Du-(Du)_{B_j}\big|^{p'}\dx\right)^{1/p'},
\end{equation}
the latter being usually called the {\em excess functional}, and 
\begin{equation}\label{mu.recall}
\mu_j:=R_j\left(\mean{B_j}|\mu|^q\dx\right)^{1/q}\quad\text{so that we have}\quad S_{R,\delta,q}(x_0)=
\sum_{j=0}^\infty\mu_j
\end{equation}
as described in \rif{serie}. 
Note that trivially, by the definitions of $\lambda, S_{R,\delta,q}(x_0)$ and our choice of $H_2\geq1$, we have
\begin{equation}\label{meas.atom}
S_{R,\delta,q}(x_0)=\sum_{j=0}^\infty R_j\left(\mean{B_{R_j}(x_0)}|\mu|^q\dx\right)^{1/q}\leq 
\left(\frac{\delta^{2n(\tilde \ell+4p)}}{10^{10p}(\Cr{c:compa.k}+\Cr{c:compa.I1})}\right)^{p-1}\lambda^{p-1}\;.
\end{equation}
Moreover, we define for positive $j\in\N$ the quantity
\begin{equation}\label{ultima}
C_j:=\left(\mean{B_{j-1}}\big(s+|Du|\big)^{p'}\dx\right)^{1/p'}+\left(\mean{B_{j}}\big(s+|Du|\big)^{p'}\dx\right)^{1/p'}+\delta^{-n}E_j.
\end{equation} 
Note that by our choice of $H_1$ we have the estimate
\begin{equation}\label{C1}
C_1\leq \frac{4}{\delta^{n}}\left(\mean{B_R(x_0)}\big(s+|Du|\big)^{p'}\dx\right)^{1/p'}\leq
\frac{4}{\delta^{n}}\frac{\lambda}{H_1}\leq \frac{\lambda}{32}\;.
\end{equation} 
We can therefore assume the existence of an exit index, $j_e\in\N$ such that 
\begin{equation}\label{exit.time}
C_{j_e}\leq\frac{\lambda}{32} \qquad\text{and}\qquad C_j>\frac{\lambda}{32}\quad\text{for all $j>j_e$}\;.
\end{equation}
Indeed, if this would not be the case, then, due to \eqref{C1}, \eqref{pointwise} would trivially follow, since there would exist a subsequence $\{j_m\}_{m\in\N}$, such that $C_{j_m}\leq \lambda$ and hence
\[
|Du(x_0)|=\lim_{j\to \infty}a_j\leq \limsup_{m\to\infty} \,C_{j_m}\leq\lambda
\]
for the Lebesgue point $x_0$ of $Du$ for which the sequence $B_j$ has been defined (such an $x_0$ being chosen arbitrarily). Observe that, from \rif{C1} it in particular follows that
\begin{equation}\label{esse}
s \leq \frac{\lambda}{10^n} \leq \frac{\lambda}{32}\;.
\end{equation}
 
\subsection{An iterative excess reduction}\label{simplify}
For $j\geq j_e$, let us consider the condition:
\begin{equation}\label{2scales}
\left(\mean{B_{j-1}}\big(s+|Du|\big)^{p'}\dx\right)^{1/p'}+\left(\mean{B_{j}}\big(s+|Du|\big)^{p'}\dx\right)^{1/p'}\leq \lambda
\end{equation}
and prove the following conditional inequality:
\begin{eqnarray}\label{sommadopo}
 &&\hspace{-4mm} \text{\eqref{2scales}}_j\quad \mbox{for} \quad  j\geq j_e\qquad\text{and}\qquad\eqref{esse}\qquad\Longrightarrow\\
 &&\nonumber\hspace{7mm} \big[E_{j+1}\big]^{p'}\leq\frac1{4^{p'}} \big[E_j\big]^{p'}+\frac{32\Cr{c:compa.I1}^{p'}}{\delta^{n}}\left\{[\tilde\omega(R_{j-1})]^{p'/2}\lambda^{p'}+\mu_{j-1}^{p'}\lambda^{p'(2-p)}\right\}.
  \end{eqnarray}
  In order to apply Lemma \ref{Lemma:comparison.I1}, we will prove that
\begin{equation}\label{point.comparison.map}
\frac\lambda{B}\leq |Dv_{j-1}|\leq B\lambda\qquad\text{in $B_j$}
\end{equation}
holds for $B=12^n\Cr{sup.est}$, that is for the choice made in \rif{delta.B}. Now we note that using our choices in
 \eqref{choice.R} and \rif{meas.atom} in estimate \eqref{compa.k}, after some computations we find
\begin{equation}\label{mean.small}
\mean{\tilde B_{j-1}}|Du-Dv_{j-1}|^{p'}\dx\leq \frac{\delta^{n(\tilde \ell+4p)}\lambda^{p'}}{10^{10p}}\leq \lambda^{p'}\;.
\end{equation}
Notice that by \rif{meas.atom} we are using that assumption \rif{ass.meas} is satisfied with the choice 
\[
A=H_2=\frac{10^{10p}(\Cr{c:compa.k}+\Cr{c:compa.I1})}{\delta^{2n(\tilde \ell+4p)}}\;.
\]
Estimate \rif{mean.small} in turn gives
\begin{multline}\label{ave.j-1}
\left(\mean{\tilde B_{j-1}}\big(s+|Dv_{j-1}|\big)^{p'}\dx\right)^{1/p'}\leq 2^{n/p'}\left(\mean{B_{j-1}}\big(s+|Du|\big)^{p'}\dx\right)^{1/p'}\\+\left(\mean{\tilde B_{j-1}}|Du-Dv_{j-1}|^{p'}\dx\right)^{1/p'} \leq 3^{n}\lambda
\end{multline}
where we have also used \eqref{2scales}. Hence, using \eqref{sup.lambda} we infer
\eqn{impl.exv.Dv_j0}
$$
\sup_{\tilde B_{j-1}/2}\, (s+|Dv_{j-1}|)\leq \Cr{sup.est}\mean{\tilde B_{j-1}}\big(s+|Dv_{j-1}|\big)\dx \leq 3^n\Cr{sup.est}\lambda
$$
so that \eqref{osc.lambda} with $B_R\equiv \tilde B_{j-1}$ implies
\eqn{impl.exv.Dv_j}
$$
\osc_{B_{j-1+\ell}}\, Dv_{j-1}\leq 3^n\Cr{sup.est}\Cr{osc.est}\delta^{\beta\ell}\lambda\qquad
\text{for all $\ell\in\N$}\;.
$$
Using first triangle's inequality, several times \eqref{prop.exc}, \eqref{impl.exv.Dv_j} and also \eqref{mean.small} we find
\begin{align}\label{exc.first.estimate}
\delta^{-p'n}\big[E_{j-1+\tilde \ell}\big]^{p'}&=\delta^{-p'n}\mean{B_{j-1+\tilde\ell}}\big|Du-(Du)_{B_{j-1+\tilde\ell}}\big|^{p'}\dx\nonumber\\
&\leq 4^{p'}\delta^{-p'n}\mean{B_{j-1+\tilde\ell}}\big|Dv_{j-1}-(Dv_{j-1})_{B_{j-1+\tilde\ell}}\big|^{p'}\dx\nonumber\\
&\hspace{2cm}+4^{p'}\delta^{-p'n}\mean{B_{j-1+\tilde\ell}}\big|Du-Dv_{j-1}\big|^{p'}\dx\nonumber\\
&\leq 4^{p'}\delta^{-p'n}\big[\osc_{B_{j-1+\tilde\ell}}\, Dv_{j-1} \big]^{p'}\nonumber\\
&\hspace{2cm}+4^{p'}\delta^{-n(p'+\tilde\ell)}\mean{\tilde B_{j-1}}\big|Du-Dv_{j-1}\big|^{p'}\dx\nonumber\\
&\leq (16^n\Cr{sup.est}\Cr{osc.est})^{p'}\delta^{(\beta\tilde\ell-n)p'}\lambda^{p'}+\frac{\delta^n\lambda^{p'}} {10^{10p}}\leq \frac{\lambda^{p'}}{10^{10p}}\;.
\end{align}
 Now, recalling that $C_{j-\tilde\ell}>\lambda/32$, we infer that 
\[
\left(\mean{B_{j-1+\tilde \ell}}\big(s+|Du|\big)^{p'}\dx\right)^{1/p'}+\left(\mean{B_{j+\tilde \ell}}\big(s+|Du|\big)^{p'}\dx\right)^{1/p'}\geq\frac{\lambda}{48}
\]
and at this point, using triangle's inequality and \eqref{mean.small}, and again recalling that $B_j\subset \tilde B_{j-1}$ we have
\begin{align*}
\frac{\lambda}{48}&\leq \left(\mean{B_{j-1+\tilde \ell}}\big(s+|Dv_{j-1}|\big)^{p'}\dx\right)^{1/p'}+\left(\mean{B_{j+\tilde \ell}}\big(s+|Dv_{j-1}|\big)^{p'}\dx\right)^{1/p'}\\
&\hspace{3cm}+2\delta^{-\frac n{p'}( \tilde \ell+1)}\left(\mean{\tilde B_{j-1}}|Du-Dv_{j-1}|^{p'}\dx\right)^{1/p'}\\
&\leq \left(\mean{B_{j-1+\tilde \ell}}\big(s+|Dv_{j-1}|\big)^{p'}\dx\right)^{1/p'}+\left(\mean{B_{j+\tilde \ell}}\big(s+|Dv_{j-1}|\big)^{p'}\dx\right)^{1/p'}\\ & \qquad 
+\frac\lambda{10^{10}}\;.
\end{align*}
The last estimate implies that at least one of the following inequalities holds:
\[
\mean{B_{j-1+\tilde \ell}}\big(s+|Dv_{j-1}|\big)^{p'}\dx\geq\frac{\lambda^{p'}}{128}\quad\text{or}\quad\mean{B_{j+\tilde \ell}}\big(s+|Dv_{j-1}|\big)^{p'}\dx\geq\frac{\lambda^{p'}}{128}\;.
\]
In any case thanks to \eqref{impl.exv.Dv_j0} - recall that $B_j\subset \tilde B_{j-1}$ - and the previous display we can apply Lemma \ref{Lemmatau} with $B_R\equiv \tilde B_{j-1}$, $b=s$, $\tau = 2\delta\leq \bar \tau$, $B_{\tau R}=B_j$ and suitable $\sigma$ depending on the size of $\tilde \ell$, which finally yields \eqref{point.comparison.map}; at this point we have the comparison estimate \rif{compa.I1} at our disposal. Using this estimate we can finally complete the proof by estimating similarly to \eqref{exc.first.estimate}. Indeed, 
using \eqref{exc.smallness} (with $\sigma =2\delta\leq \bar \delta$, recall \rif{delta.B}) together with our choice of $\bar\eps$ (made after \rif{delta.B}) and \eqref{compa.I1}:
\begin{align}\label{red.exc}
\nonumber\big[E_{j+1}\big]^{p'} &\leq 4^{p'}\mean{B_{j+1}}\big|Dv_j-(Dv_j)_{B_j}\big|^{p'}\dx+4^{p'}\mean{B_{j+1}}|Du-Dv_j|^{p'}\dx\\
&\leq \frac1{4^{8p}}\mean{\tilde B_j}\big|Dv_j-(Dv_j)_{B_j}\big|^{p'}\dx+\frac{4^{p'}}{\delta^{n}}\mean{\tilde B_j}|Du-Dv_j|^{p'}\dx\nonumber\\
&\leq \frac1{4^{5p}}\mean{B_j}\big|Du-(Du)_{B_j}\big|^{p'}\dx+16\big(1+\delta^{-n}\big)\mean{\tilde B_j}|Du-Dv_j|^{p'}\dx\nonumber\\
&\leq \frac1{4^{p'}} \big[E_j\big]^{p'}+\frac{32\Cr{c:compa.I1}^{p'}}{\delta^{n}}\left\{[\tilde \omega(R_{j-1})]^{p'/2}\lambda^{p'}+\mu_{j-1}^{p'}\lambda^{p'(2-p)}\right\}
\end{align}
and \eqref{sommadopo} is finally proven.
\subsection{An iterative mean-value bound}
Here we prove, by induction, that
\begin{equation}\label{final.size.estimate}
s+a_j+E_j\leq \frac\lambda2 ,\qquad \mbox{holds for every} \ j\geq j_e
\end{equation}
and this will immediately imply the pointwise estimate \eqref{pointwise} as in fact 
\[
\lim_{j \to \infty} a_j =|Du(x_0)|
\]
for every Lebesgue point $x_0$ of $Du$. Note that \eqref{final.size.estimate}$_{j_e}$ holds as consequence of the true definition of the exit-time index $j_e$ in \eqref{exit.time}. Indeed, simply note that
\[
s+a_{j_e}+E_{j_e}  \leq 2\left(\mean{B_{j_e}}\big(s+|Du|\big)^{p'}\dx\right)^{1/p'} + \delta^{-n}E_{j_e} \leq C_{j_e} \leq \frac{\lambda}{32}\;.
\]
Next, we suppose that, for some $\bar\jmath\geq j_e$, \eqref{final.size.estimate}$_j$ holds true for $j=j_e,\dots,\bar\jmath$. 
We note that 
\[
\eqref{final.size.estimate}_{j-1} \ \mbox{and} \ \eqref{final.size.estimate}_j\Longrightarrow \eqref{2scales}_j
\]
for all $j>j_e$ since we can use triangle inequality to estimate 
\begin{align*}
& \left(\mean{B_{j-1}}(s+|Du|)^{p'}\dx\right)^{1/p'}+\left(\mean{B_j}(s+|Du|)^{p'}\dx\right)^{1/p'}\\& 
\hspace{45mm}\leq 2s+  a_{j-1}+E_{j-1}+a_j+E_j\leq \lambda\;.
\end{align*}
Therefore \eqref{2scales}$_j$ holds for $j=j_e,\dots,\bar\jmath$. Indeed, if $j=j_e$ this plainly follows by the definition of $C_{j_e}$, 
while if $j>j_e$ this is essentially the content of the second-last display. Hence \eqref{sommadopo}$_j$ is at our disposal for $j=j_e,\dots,\bar\jmath$; summing up yields
\[
\sum_{j=j_e+1}^{\bar \jmath+1}E_j\leq\frac14 \sum_{j=j_e}^{\bar \jmath}E_j+\frac{32\Cr{c:compa.I1}}{\delta^{n}}
\sum_{j=j_e}^{\bar \jmath}\left\{[\tilde \omega(R_{j-1})]^{1/2}\lambda+\mu_{j-1}\lambda^{2-p}\right\}\;.
\]
Reabsorbing part of the first sum on the right-hand side and recalling the definition of $S_{R,\delta,q}(x_0)$ (see for instance \eqref{mu.recall}) we then have 
\begin{align}\label{E_j+1}
\sum_{j=j_e}^{\bar \jmath+1}E_j& \leq2E_{j_e}+\frac{4^4\Cr{c:compa.I1}}{\delta^{n}}
\left\{\sum_{j=0}^{\infty}[\tilde \omega(R_j)]^{1/2}+S_{R,\delta,q}(x_0)\lambda^{1-p}\right\}\lambda\nonumber\\
&\leq2\delta^n C_{j_e}+
\frac{4^4\Cr{c:compa.I1}}{\delta^{n}}\left\{\delta^{-1}\int_0^{4R}[\tilde \omega(\rho)]^{1/2}\frac{d\rho}{\rho}+S_{R,\delta,q}(x_0)\lambda^{1-p}\right\}\lambda\nonumber\\
&\leq \frac{\delta^n\lambda}{16}+\frac{\delta^n\lambda}{16}=\frac{\delta^n\lambda}{8}
\end{align}
by \eqref{choice.R} and \eqref{meas.atom}. We made use of the estimate
\[
\delta\sum_{j=0}^\infty[\tilde \omega(R_j)]^{1/2}\leq \int_0^{4R}[\tilde\omega(\rho)]^{1/2}\frac{d\rho}{\rho}\leq \int_0^{4R_0}[\tilde\omega(\rho)]^{1/2}\frac{d\rho}{\rho};
\]
see for instance \cite[Equation ($88$)]{KMStein}. In particular, we have proved that
\[
E_{\bar \jmath+1} \leq \frac{\lambda}{64}\;.
\]
At this point we infer
\begin{align*}
a_{\bar\jmath+1}&=a_{j_e}+\sum_{j=j_e}^{\bar \jmath}[a_{j+1}-a_j]\leq a_{j_e}+\sum_{j=j_e}^{\bar \jmath+1}\mean{B_{j+1}}|Du-(Du)_{B_j}|\dx\\
&\leq a_{j_e}+\delta^{-n}\sum_{j=j_e}^{\bar \jmath+1}E_j\leq C_{j_e}+\frac{\lambda}{8}\leq \frac\lambda4\;.
\end{align*}
Using the content of the last two displays and \rif{esse} yields
\[
s+a_{\bar\jmath+1}+E_{\bar\jmath+1}\leq  \frac{\lambda}{32} + \frac{\lambda}{4} + \frac\lambda{64} \leq \lambda\;.
\]
and the inductive step is concluded. This also concludes the proof of \rif{pointwise} thereby establishing the local boundedness of $Du$. 

\subsection{The case $1<p<2$ and Theorem \ref{second.thm}}\label{rem sec} The proof of Theorem \ref{main.thm} in the case $1<p<2$
is very much along the lines of the one for the case $p \geq 2$, actually being much simpler. 
The proof can be deduced from the one in the previous section upon replacing, in the definitions given through \rif{pointwise}-\rif{ultima} - the exponent $p'$ by $p$. In particular, the new definitions of $E_j$ and $C_j$ can be now given as
\[
E_j:=\left(\mean{B_j}\big|Du-(Du)_{B_j}\big|^{p}\dx\right)^{1/p}
\]
and
\[
C_j:=\left(\mean{B_{j-1}}\big(s+|Du|\big)^{p}\dx\right)^{1/p}+\left(\mean{B_{j}}\big(s+|Du|\big)^{p}\dx\right)^{1/p}
+\delta^{-n}E_j
\]
respectively. The proof then proceeds in the same way, upon using Lemma \ref{comparison.I1-} instead of Lemma \ref{Lemma:comparison.I1}. In this respect we just notice that the content of Section \ref{simplify} can be simplified in that we actually do not need to check the validity \rif{point.comparison.map}, since this is not required in Lemma \ref{comparison.I1-}.  

As for Theorem \ref{second.thm}, we notice that the proofs are now exactly equivalent, upon using Lemmas \ref{Lemma:improved.comparison.I1}-\ref{improvedcomparison.I1-} instead of Lemmas \ref{Lemma:comparison.I1}-\ref{comparison.I1-}. The only difference 
is essentially in that, whenever it appears, the modulus of continuity $\tilde \omega(\cdot)$ must be replaced by $[\bar \omega(\cdot)]^2$

\section{Proof of Theorems \ref{main.thm}-\ref{second.thm}: gradient continuity}\label{finalsec}
By the result obtained in the previous section, and since all the statements we are going to prove are local in nature, we can then assume that $Du \in L^\infty(\Omega)$
and therefore put
\begin{equation}\label{lambda.continuity}
\lambda:=4\|Du\|_{L^\infty(\Omega)}+4\;.
\end{equation}
Now, with $\Omega'\Subset\Omega$ being fixed, we want to prove that the gradient is 
continuous in $\Omega'$ by showing that $Du$ is the uniform limit of a net of continuous functions, namely its averages on small balls. In particular we are going to show that, for any $\eps\in(0,1)$, there exists a radius $R_\eps\leq \dist(\Omega',\partial\Omega)/100$ such that
\begin{equation}\label{quasi.Cauchy}
\big|(Du)_{B_{r_1}(x_0)}-(Du)_{B_{r_2}(x_0)}\big|\leq \eps\lambda
\end{equation}
for any $0<r_1<r_2\leq R_\eps$, uniformly for $x_0\in\Omega'$. This
means that the limit
\[
\lim_{\varrho \to 0}\,(Du)_{B_{\varrho}(x_0)}=Du(x_0)
\]
is locally uniform with respect to $x_0\in \Omega'$ and therefore the gradient $Du$ (which is in fact pointwise defined by 
the previous equality in the sense of the usual precise representative) is continuous. We are of course using that the functions $x_0 
\mapsto (Du)_{B_{\varrho}(x_0)}$ are continuous for fixed $\varrho \in (0,{\rm dist}(\Omega', \partial \Omega)/100)$. 

The proof uses a few of the arguments developed in the preceding section to get the local boundedness of the gradient and for this reason we shall confine ourselves to give the proof of Theorem \ref{main.thm} in the case $p\geq 2$. The proof for the case $1< p < 2$ and the one of Theorem \ref{second.thm} can be obtained with modifications similar to those described in Section \ref{rem sec}. We therefore start considering the numbers
\begin{equation}\label{delta.B.eps}
B=\frac{12^n\Cr{sup.est}}\eps,\qquad  \delta:=\min\{\bar\tau,\bar\delta\}/8\in(0,1/4)\;.
\end{equation}
Here
\begin{itemize}
\item $\Cr{sup.est}$ is the constant appearing in the $\sup$ estimate of Lemma \ref{Lem.regolarity}
\item $\bar\tau$ is the constant appearing in \rif{sigma} from Lemma \ref{Lemmatau} with the choice 
$\tilde A= 10^n\Cr{sup.est}/\eps$
\item $\eps$ is the number defined in the inequality in \rif{quasi.Cauchy}, that we want to prove here 
\item $\bar\delta$ is provided by Lemma \ref{ex.deca} for the choice $\bar\eps=4^{-10p}\eps$.
\end{itemize}
Notice that in this way both $B$ and $\delta$ are determined as functions of the fixed parameters $n,p,\nu, L, \eps$. Now we fix a radius $R_{1,\eps}\leq \dist(\Omega',\partial\Omega)/100$ such that the following conditions are satisfied: 
\begin{align}\label{choice.R1}
\nonumber 
&\tilde \omega(R_{1,\eps})\leq \frac{\delta^{4np}}{(\Cr{c:compa.k}+\Cr{c:compa.I1})^{2p}}\Big(\frac{\eps}{200}\Big)^{2p}\\
\nonumber
&\delta^{1-n/q}d(R_{1,\eps}) \leq \left[\frac{\delta^{4n}}{10^{4n}(\Cr{c:compa.k}+\Cr{c:compa.I1})}\right]^{p}\Big(\frac{\eps}{200}\Big)^{p}\\
&\delta^{-1}\int_0^{4R_{1,\eps}}[\tilde \omega(\rho)]^{1/2}\frac{d\rho}{\rho} \leq \frac{\delta^{3n}}{8\Cr{c:compa.I1}}\frac\eps{100}
\end{align}
where 
\begin{itemize}
\item the function $d(\cdot)$ has been introduced in Lemma \ref{Lem.Ln1}
\item the constant $\Cr{c:compa.k}\equiv \Cr{c:compa.k}(n,p,\nu, L,\tilde L,\|\mu\|_{L^n})$ appears in Lemma \ref{Lemma:comparison.I1}
\item the constant $\Cr{c:compa.I1}$ appears in Lemma \ref{Lemma:comparison.I1} with the choice $B =12^nc_1/\eps $. In this way we have that $c_4\equiv c_4(n,p,\nu, L,\tilde L,\|\mu\|_{L^n}, \eps)$
\item the exponent $q\equiv q(n,p) \in (1,n)$ has been defined in \eqref{fissati}. 
\end{itemize}
Let us observe that the radius $R_{1,\eps}$ depends only on $n,p,\nu,L,\omega(\cdot), d(\cdot), \dist(\Omega',\partial\Omega)$ and $\eps$. The dependence on $\mu$ is incorporated in the one on $d(\cdot)$. Using Lemma \ref{Lem.Ln1}, we then find that the following inequality holds too:
\begin{equation}\label{choice.R1d}
\sup_{0<\rho\leq R_{1,\eps}}\sup_{x_0\in\Omega}S_{\rho,\delta,q}(x_0) \leq\left[\frac{\delta^{4n}}{10^{4n}(\Cr{c:compa.k}+\Cr{c:compa.I1})}\right]^{p}\Big(\frac{\eps}{200}\Big)^{p} \;. 
\end{equation}
Then we fix a {\em starting radius} 
\begin{equation}\label{ilraggio}
R\in( \delta R_{1,\eps},R_{1,\eps}]
\end{equation} and, for points $x_0\in\Omega'$, we define radii and balls as follows:
\[
R_j:=\delta^j R,\qquad\qquad B_j:=B_{R_j}(x_0)\;.
\]
Accordingly $E_j\equiv E_j(x_0)$ and $\mu_j$ are defined as in \eqref{mu.recall0} and \eqref{mu.recall}, respectively. 
Note that \eqref{choice.R1d} implies that \rif{ass.meas} holds for instance with 
\[A=\frac{200\Cr{c:compa.k}}{\eps\delta^{2np}}\;.\]
Finally, instead of using an exit time argument, we shall consider the following condition:
\begin{equation}\label{exit.continuity}
\left(\mean{B_{j+1}}|Du|^{p'}\dx\right)^{1/p'}\geq\frac{\eps\lambda}{50}\;.
\end{equation}
We remark that all the forthcoming estimates are independent of the point $x_0\in \Omega'$. 
\subsection{Reduction of the excess, again}\label{red.excess}
Here we show the validity of the implication
\begin{eqnarray}
\nonumber &&\hspace{-4mm}
 \text{\eqref{exit.continuity}}_j\quad\Longrightarrow\\
&& \big[E_{j+1}\big]^{p'}\leq\Big(\frac\eps4\Big)^{p'} \big[E_j\big]^{p'}+\frac{32\Cr{c:compa.I1}^{p'}}{\delta^{n}}\left\{[\tilde \omega(R_{j-1})]^{p'/2}
\lambda^{p'}+\mu_{j-1}^{p'}\lambda^{p'(2-p)}\right\}\,.\label{exc.reduction.continuity}
\end{eqnarray}
To start with the proof we infer by \eqref{compa.k}, \eqref{choice.R1}$_{1}$ and the definition of $A$ the following inequality:
\begin{equation}\label{comp.j-1.continuity}
\mean{\tilde B_{j-1}}|Du-Dv_{j-1}|^{p'}\dx\leq\delta^{2n} \Big(\frac{\eps\lambda}{200}\Big)^{p'}+\delta^{2n} \Big(\frac{\eps\lambda}{200}\Big)^{p'}\leq \lambda^{p'}  
\end{equation}
which implies, by the definition of $\lambda$ in \eqref{lambda.continuity} and in a manner which is absolutely analogous to \eqref{ave.j-1},
\[
\left(\mean{\tilde B_{j-1}}(s+|Dv_{j-1}|)^{p'}\dx\right)^{1/p'}\leq  3^n\lambda
\]
and therefore, by Lemma \ref{Lem.regolarity} we conclude with
\[
\sup_{\tilde B_{j-1}/2}(s + |Dv_{j-1}|) \leq  3^n\Cr{sup.est}\lambda\;.
\]
Then, using \eqref{exit.continuity}$_j$ and \eqref{comp.j-1.continuity}
\begin{align*}
\frac{\eps\lambda}{50} &\leq \left(\mean{B_{j+1}}|Du|^{p'}\dx\right)^{1/p'}\\ &\leq \left(\mean{B_{j+1}}|Dv_{j-1}|^{p'}\dx\right)^{1/p'}+\delta^{-2\frac n{p'}}\left[\delta^{2n} \Big(\frac{\eps\lambda}{100}\Big)^{p'}\right]^{1/p'}
\end{align*}
that is, 
\[
\frac{\eps\lambda}{10^n\Cr{sup.est}}\leq\frac{\eps\lambda}{100}\leq \left(\mean{B_{j+1}}|Dv_{j-1}|^{p'}\dx\right)^{1/p'}\]
and 
\[\sup_{\tilde B_{j-1}/2}(s+|Dv_{j-1})|\leq  \frac{10^n\Cr{sup.est}\lambda}{\eps}\;.
\]
Lemma \ref{Lemmatau} with $b=0$ then gives
\[
|Dv_{j-1}|\geq \frac{\eps\lambda}{12^n\Cr{sup.est}}\quad \text{in $B_j$}
\]
and therefore, all in all we have proved that
\[\frac{\eps\lambda}{12^n\Cr{sup.est}}\leq |Dv_{j-1}|\leq  \frac{12^n\Cr{sup.est}\lambda}{\eps}\quad \text{in $B_j$.}
\]
At this point we can use Lemma \ref{Lemma:comparison.I1} (recall \eqref{lambda.continuity} and that \eqref{ass.meas} is satisfied due to \eqref{choice.R1d}) that gives
\[
\mean{\tilde B_j}|Du-Dv_j|^{p'}\dx\leq \Cr{c:compa.I1}^{p'}[\tilde \omega(R_{j-1})]^{p'/2}\lambda^{p'}+\Cr{c:compa.I1}^{p'}\lambda^{p'(2-p)}\mu_{j-1}^{p'}
\]
and following \eqref{red.exc} (the only difference at this point is the presence of a multiplicative factor $\eps^{p'}$ when using \eqref{exc.smallness}) we get \eqref{exc.reduction.continuity}. We remark that all the above estimates are independent of the choice of the starting radius $R$ fixed in \rif{ilraggio} and of the point $x_0\in\Omega'$. This fact will be used in the following step of the proof. 

\subsection{Smallness of the excess}
We prove in this section the following fact: for every $\sigma\in(0,1)$, there exists a radius $R_{2,\sigma}$, depending only on 
$n,p,\nu,L,\omega(\cdot), d(\cdot)$ and $\sigma$ such that 
\begin{equation}\label{prima}
\sup_{x_0\in\Omega'}\sup_{\rho\in(0,R_{2,\sigma}]}\mean{B_\rho(x)}\big|Du-(Du)_{B_\rho(x)}\big|^{p'}\dx\leq \sigma\lambda^{p'}.
\end{equation}
Notice that this will follow with the choice $R_{2,\sigma}:=\delta^2R_{1,\eps}$, where $\delta$ and $R_{1,\eps}$ have been fixed in \rif{delta.B.eps} and \eqref{choice.R1}, respectively, for the choice $\eps:=\sigma^{1/p'}/4$, once we prove that
\begin{equation}\label{iiiiis}
\big[E_{j+1}(x_0)\big]^{p'}\leq \sigma \lambda^{p'}\qquad\text{for any }j\in\N
\end{equation}
holds uniformly for $x_0\in\Omega'$, with the notation of the previous section. This is in turn a consequence of a simple argument together with \eqref{prop.exc}: indeed, take $r\leq R_{2,\sigma}:=\delta^2R_{1,\eps}$. Then $r=\delta^m R$ for some $m\in\N$, $m\geq2$ and some $R\in (\delta R_{1,\eps},R_{1,\eps}]$ and \rif{prima} and \eqref{iiiiis} coincide, since $R_j=r$. It therefore remains to prove  \eqref{iiiiis} and this is not difficult once having the estimates of Section \ref{red.excess} at our disposal; these will applied with the choice $\eps:=\sigma^{1/p'}/4$. Indeed, take $x\in\Omega'$ and consider the condition \eqref{exit.continuity} for $\eps:=\sigma^{1/p'}/4$. If it does not hold, then we have
\[
\big[E_{j+1}(x_0)\big]^{p'}\leq 2^{p'}\mean{B_{j+1}}|Du|^{p'}< 2^{p'}\Big(\frac{\eps}{50}\Big)^{p'}\lambda^{p'}\leq \sigma\lambda^{p'}.
\]
If on the other hand condition \eqref{exit.continuity} does hold, then we can use the inequality in \eqref{exc.reduction.continuity} together with \eqref{lambda.continuity}:
\begin{align*}
\big[E_{j+1}(x_0)\big]^{p'}&\leq\frac\sigma{16^{p'}} \big[E_j(x_0)\big]^{p'}+\frac{32\Cr{c:compa.I1}^{p'}}{\delta^{n}}\left\{[\tilde \omega(R_{j-1})]^{p'/2}
\lambda^{p'}+\mu_{j-1}^{p'}\lambda^{p'(2-p)}\right\}\\
&\leq \frac{\sigma}{4}\lambda^{p'}+\frac{\sigma}{4}\lambda^{p'}+\frac{\sigma}{2}\lambda^{p'}=\sigma \lambda^{p'},
\end{align*}
where we have also used \eqref{choice.R1}$_{1}$ and \eqref{choice.R1d}.

\subsection{A Cauchy-type sequence} By the result of the last section, for every $\eps\in(0,1)$ and with 
$\delta \equiv \delta (n,p,\nu, L, \eps)$ accordingly defined as in \rif{delta.B.eps}, we can find a radius $R_{3,\eps}$, depending on $n,p,\nu,L,\omega(\cdot), d(\cdot)$ and $\eps$, such that 
\begin{equation}\label{sm.ex.uniform}
\sup_{x\in\Omega'}\sup_{\rho\in(0,R_{3,\eps}]}\mean{B_\rho(x)}\big|Du-(Du)_{B_\rho(x)}\big|^{p'}\dx\leq \delta^{2np'}\Big(\frac{\eps}{200}\Big)^{p'}\lambda^{p'}\;. 
\end{equation}
We take $R_{3,\eps}$ in such a way that also $R_{3,\eps}\leq R_{1,\eps}$ holds, so that also \rif{choice.R1}-\rif{choice.R1d} are at our disposal. We then finally prove that \eqref{quasi.Cauchy} follows with $R_\eps:=\delta^2R_{3,\eps}$. We re-define the sequences 
\[
\bar R_j:=\delta^j R_{3,\eps},\qquad\qquad \bar B_j:=B_{\bar R_j}(x_0)
\]
and accordingly $\bar E_j$ and $\bar \mu_j$ are defined as in \eqref{mu.recall0} and \eqref{mu.recall}, respectively, 
starting from the new balls $\bar B_j$ instead of $B_j$. Let us now prove
\begin{equation}\label{discrete.quasi.Cauchy}
 \big|(Du)_{\bar B_k}-(Du)_{\bar B_h}\big|\leq\frac{\eps}{12}\lambda\qquad\text{for $2\leq h< k$}
\end{equation}
and later we shall show how to deduce \eqref{quasi.Cauchy} from the previous estimate. Define
\[
\mathcal L:=\left\{j\in\N_0:\left(\mean{\bar B_j}|Du|^{p'}\dx\right)^{1/p'}\leq\frac{\eps}{50}\lambda\right\}\qquad\text{and}\qquad j_e:=\min \mathcal L;
\]
note that it can be $\mathcal L=\emptyset$; in this case $j_e=\infty$. Now we deduce the following auxiliary estimate, useful in few lines: take $2\leq i_1\leq i_2$ and suppose $i\not\in\mathcal L$ for every $i\in\{i_1,\dots,i_2\}$. Then  \eqref{exit.continuity}$_j$ holds for $j=i_1-1,\dots,i_2-1$ and, for the same indexes, we have that the excess decay inequality in \eqref{exc.reduction.continuity} holds; summing up over the indexes, recalling that $\eps\in(0,1)$ and performing algebraic manipulations similar to those made in order to get \eqref{E_j+1} gives, since $R_{3,\eps}\leq R_{1,\eps}$
\begin{align}\label{small.sum.excess}
\sum_{j=i_1-1}^{i_2}\bar E_j&\leq 2\bar E_{i_1-1}+\frac{4^4\Cr{c:compa.I1}}{\delta^n}
\left\{\sum_{j=0}^\infty[\tilde \omega(\bar R_j)]^{1/2}+S_{R_{1,\eps},\delta,q}(x)\lambda^{1-p}\right\}\lambda\nonumber\\
&\leq\delta^{2n}\frac{\eps}{100}\lambda+\delta^{2n}\frac{\eps}{200}\lambda+\delta^{2n}\frac{\eps}{200}\lambda=\delta^{2n}\frac{\eps}{50}\lambda\;.
\end{align}
Notice that we have used \rif{sm.ex.uniform} and \rif{choice.R1}-\rif{choice.R1d}. 
Now we distinguish  three different cases: the first one is when $2\leq h< k< j_e$; this means that we can use \eqref{small.sum.excess} with $i_1=h$, $i_2=k$, yielding almost immediately \eqref{discrete.quasi.Cauchy}:
\begin{align}\label{sum.example}
\nonumber \big|(Du)_{\bar B_k}-(Du)_{\bar B_h}\big|
& \leq \sum_{j=h}^{k-1}\big|(Du)_{\bar B_{j+1}}-(Du)_{\bar B_j}\big|\\
& \leq \sum_{j=h}^{k-1}\mean{\tilde B_{j+1}}\big|Du-(Du)_{\bar B_j}\big|\dx\\
\nonumber
&\leq \delta^{-n}\sum_{j=h}^{k-1}\left(\mean{\bar B_j}\big|Du-(Du)_{\bar B_j}\big|^{p'}\dx \right)^{1/p'}\\& 
=\delta^{-n}\sum_{j=h}^{k-1}\bar E_j\leq \frac{\eps}{25}\lambda\;.
\end{align}
The second case is when $j_e\leq h< k$, and here we shall prove that
\begin{equation}\label{stupid.smallness}
\begin{array}{c}\displaystyle 
\big|(Du)_{\bar B_k}|\leq \frac{\eps}{25}\lambda\\ [10 pt]
\displaystyle |(Du)_{\bar B_h}\big|\leq \frac{\eps}{25}\lambda 
\end{array}
\end{equation}
from which clearly \eqref{discrete.quasi.Cauchy} follows; we start by estimating the first term. If $k\in\mathcal L$, \eqref{stupid.smallness}$_1$ is trivial by the definition of $\mathcal L$. If $k\not\in\mathcal L$, then we can consider the chain $\mathcal C:=\{\bar\jmath,\dots,k\}\subset\{j_e,\dots,k\}$ such that $\bar\jmath\in\mathcal L$ and $\{\bar\jmath+1,\dots,k\}\cap\mathcal L=\emptyset$; clearly, such chain exists since $k>j_e$. Again, we shall use \eqref{small.sum.excess} for $i_1=\bar\jmath+1$ and $i_2=k$ and, since $\bar\jmath\in\mathcal L$, we have by triangle's inequality
\[
\big|(Du)_{\bar B_k}\big|\leq \big|(Du)_{\bar B_{\bar\jmath}}\big|+\sum_{j=\bar\jmath}^{k-1}\big|(Du)_{\bar B_{j+1}}
-(Du)_{\bar B_j}\big|\leq \frac{\eps}{50}\lambda+\sum_{j=\bar\jmath}^{k}\tilde E_j\leq \frac{\eps}{25}\lambda
\]
similarly to what done in \eqref{sum.example}. The estimate for \eqref{stupid.smallness}$_2$ is analogous. Finally, we consider the case where $h< j_e\leq k$: as in \eqref{stupid.smallness}, $|(Du)_{\tilde B_k}|\leq {\eps\lambda}/{25}$ while, by \eqref{sum.example} with $k=j_e$
\[
\big|(Du)_{\bar B_{j_e}}-(Du)_{\bar B_h}\big|\leq\frac{\eps}{50}\lambda\ \ \Longrightarrow\ \ \big|(Du)_{\bar B_h}\big|\leq \big|(Du)_{\bar B_{j_e}}\big|+\frac{\eps}{50}\lambda\leq\frac{\eps}{25}\lambda\;.
\]
To conclude, we show how to deduce \eqref{quasi.Cauchy} from \eqref{discrete.quasi.Cauchy}. Indeed, take any two radii $0<r_1<r_2\leq R_\eps=\delta^2R_{3,\eps}$. Then, there are two exponents $2\leq h\leq k$ such that $\delta^{k+1}R_{3,\eps}\leq r_1\leq \delta^kR_{3,\eps}$ and $\delta^{h+1}R_{3,\eps}\leq r_2\leq \delta^hR_{3,\eps}$; using H\"older's inequality
\begin{multline*}
 \big|(Du)_{B_{r_1}}-(Du)_{\bar B_{k+1}}\big|\leq \mean{\bar B_{k+1}}\big|Du-(Du)_{B_{r_1}}\big|\dx\\
\leq \delta^{-n} \mean{B_{r_1}}\big|Du-(Du)_{B_{r_1}}\big|\dx\leq \frac\eps{10}\lambda
\end{multline*}
by \eqref{sm.ex.uniform}; analogously for $|(Du)_{B_{r_2}}-(Du)_{\bar B_{h+1}}|\leq \eps\lambda/10$; connecting these two estimates with \eqref{discrete.quasi.Cauchy} yields \eqref{quasi.Cauchy}. The proof is complete (see 
the remarks at the beginning of Section \ref{finalsec} for the case $1<p<2$ and the proof of Theorem \ref{second.thm}). 

\vs

{\bf Acknowledgments:}
This work has been supported by the Academy of Finland project 
``Regularity theory for nonlinear parabolic partial differential equations'' and by GNAMPA-INDAM. 

\end{document}